\documentclass{article}

\usepackage[utf8]{inputenc}
\usepackage[english]{babel}
\usepackage{csquotes}
\usepackage{import}

\usepackage{mathrsfs}
\usepackage{enumerate}
\usepackage{amsthm}
\usepackage{thmtools}
\usepackage{thm-restate}
\usepackage{amsmath}
\usepackage{amsfonts}
\usepackage{amssymb}
\usepackage{dsfont}
\usepackage{stmaryrd}
\usepackage{esint}
\usepackage{mathtools}
\usepackage{xfrac}
\usepackage{float}
\usepackage{multicol}
\usepackage[dvipsnames]{xcolor}

\usepackage[font=scriptsize]{caption}
\usepackage[colorlinks,
            linkcolor=blue,
            anchorcolor=red,
            citecolor=green
            ]{hyperref}
\usepackage{graphicx}

\usepackage{verbatim}

\usepackage[backend=biber,isbn=false,url=false,eprint=false,maxnames=6,doi=false]{biblatex}

\DeclareFieldFormat{doi}{%
  \mkbibacro{DOI}\addcolon\space
  \ifhyperref
    {\href{https://doi.org/#1}{\nolinkurl{https://doi.org/#1}}}
    {\nolinkurl{https://doi.org/#1}}}
    
\DeclareFieldFormat[article]{volume}{Vol. #1}
\DeclareFieldFormat[thesis]{title}{\textit{#1}}
\DeclareFieldFormat[article]{issue}{\textit{#1}}

\DeclareCaptionLabelFormat{custom}{\textbf{Fig}}
\DeclareCaptionLabelSeparator{custom}{\textbf{. 1}}

\usepackage{tikz}
\usetikzlibrary{matrix}
\usetikzlibrary{decorations.pathmorphing, decorations.pathreplacing, decorations.shapes}

\numberwithin{equation}{section}

\theoremstyle{plain}

\newtheorem{theorem}{Theorem}[section]
\newtheorem*{theorem*}{Theorem}
\newtheorem{lemma}[theorem]{Lemma}
\newtheorem{corollary}[theorem]{Corollary}
\newtheorem{definition}[theorem]{Definition}
\newtheorem{remark}[theorem]{Remark}

\DeclareMathOperator{\R}{\mathbb{R}}

\DeclareMathOperator{\C}{\mathbb{C}}

\DeclareMathOperator{\N}{\mathbb{N}}
\DeclareMathOperator{\Z}{\mathbb{Z}}

\DeclareMathOperator{\Id}{Id}

\DeclareMathOperator{\supp}{supp}

\DeclareMathOperator{\Real}{Re}
\DeclareMathOperator{\Imag}{Im}

\DeclareMathOperator{\loc}{loc}

\DeclareMathOperator{\sech}{sech}

\newcommand*\dd{\mathop{}\!\mathrm{d}}

\let \phi=\varphi

\let \epsilon=\varepsilon

\let \del=\partial

\title{Global-in-time Well-posedness of the One-dimensional \protect\\ Hydrodynamic Gross-Pitaevskii Equations without Vacuum}

\author{Robert Wegner}
\date{}

\begin{document}

\maketitle

\setcounter{tocdepth}{2}
\tableofcontents

\begin{abstract}
   We establish global-in-time well-posedness of the one-dimensional hydrodynamic Gross-Pitaevskii equations in the absence of vacuum in $(1 + H^s) \times H^{s-1}$ with $s \geq 1$. We achieve this by a reduction via the Madelung transform to the previous global-in-time well-posedness result for the Gross-Pitaevskii equation in \cite{KochLiao, KochLiao2022}. Our core result is a local bilipschitz equivalence of the relevant function spaces, which enables the transfer of results between the two equations.
\end{abstract}

\vspace{.8cm}


\vspace{.4cm}

 
\section{Introduction}
We consider in one dimension the Gross-Pitaevskii equation
\begin{equation} 
    \tag{GP}
    i \del_t q + \del_{xx} q - 2 (|q|^2 - 1) q = 0 \,,
\end{equation}
where $q(t,x): \R \times \R \longrightarrow \C$ represents an unknown wave function, subject to the boundary condition at infinity $\lim_{|x| \rightarrow \infty} |q(t,x)| = 1$. In the absence of vacuum, meaning that 
\begin{align*}
    |q| > 0 \,,
\end{align*} the Gross-Pitaevskii equation has a hydrodynamic formulation
\begin{equation}
    \tag{hGP}
    \Bigg\{
    \arraycolsep=1.4pt
    \begin{array}{rl}
        \del_t \rho + 2 \del_x (\rho v) & = 0 \,, \\
        \del_t v + \del_x (v^2) + 2 \del_x \rho & = \del_x \big( \del_x \big( \frac12 \frac{\del_x \rho}{\rho} \big) + \big( \frac12 \frac{\del_x \rho}{\rho} \big)^2 \big) \,,
    \end{array}
\end{equation}
which we call the hydrodynamic Gross-Pitaevskii equations.
Here the functions $\rho(t,x): \R \times \R \longrightarrow \R_+$ and $v(t,x): \R \times \R \longrightarrow \R$ may be understood as the unknown density and velocity of a quantum fluid. The system (hGP) belongs to the class of quantum hydrodynamical models, which may be used to model various physical phenomena such as Bose‐Einstein condensation \cite{Dalfovo1998, Grant1973}, superfluidity \cite{Feynman, Landau, Loffredo1993} and quantum semiconductors \cite{Gardner}. We refer to \cite{AHMZ, Bresch2017, Hientzsch} for more information on quantum hydrodynamical models and their relation to nonlinear Schr\"odinger equations.

The relation between (GP) and (hGP) is given by the Madelung transform 
\begin{equation} \label{eqn:1}
    \mathcal{M}(q) = \Big(|q|^2, \Imag\Big[\frac{\del_x q}{q}\Big]\Big) \,,
\end{equation}
which formally transforms a solution $q$ of (GP) into a solution $(\rho, v) = \mathcal{M}(q)$ of (hGP). Note that $\rho$ and $v$ are real-valued. One immediately sees that the Madelung transform $\mathcal{M}$ only makes sense when $|q| > 0$. In this case, we may recover $q$ from its Madelung transform by the formula
\begin{equation*}
    q = \sqrt{\rho} e^{i \phi} \,,
\end{equation*}
where $\phi$ is some spatial primitive of $v$, i.e.
\begin{equation*}
    \del_x \phi = v\,.
\end{equation*}
One furthermore sees that the inverse Madelung transform $(\rho, v) \longmapsto q$ is only defined up to multiplication with $\mathbb{S}^1$, i.e. a constant rotation in phase (see (\ref{eqn:201}) below for more details).
We refer the reader to \cite{CarlesDanchinSaut} for a survey of the Madelung transform and the hydrodynamic Gross-Pitaevskii equations.

\subsection{Related results for the Gross-Pitaevskii equation}

E.P. Gross \cite{Gross} and L.P. Pitaevskii \cite{Pitaevskii} introduced the Gross-Pitaevskii equation as a model for a Bose-Einstein Condensate, a type of Boson gas at very low density and temperature. For rigorous justification of the model, we refer to the mean-field approximation established by L. Erdős, B. Schlein, and H. Yau \cite{ErdosSchleinYau}, as well as references therein. As the Gross-Pitaevskii equation is a kind of defocusing cubic nonlinear Schr\"odinger equation, its well-posedness has been extensively studied. Due to the non-zero boundary condition, finite-energy solutions to (GP) can clearly not be in traditional function spaces that require global integrability, such as $L^p(\R)$. For integers $k \geq 1$ and in any dimension $n \geq 1$, P.E. Zhidkov \cite{Zhidkov1987} established local-in-time well-posedness in the so-called Zhidkov space $Z^k(\R^n)$, which is the closure of $\{u \in C_b^k(\R^n): \del_x u \in H^{k-1}(\R^n)\}$ under the norm
\begin{equation} \label{eqn:202}
    \|u\|_{Z^k(\R^n)} = \|u\|_{L^\infty(\R^n)} + \sum_{1 \leq |\alpha| \leq k} \|\del_x^\alpha u\|_{L^2(\R^n)}.
\end{equation}
This led to a first global-in-time well-posedness result in one dimension in $Z^1(\R)$, as the Ginzburg-Landau energy 
\begin{equation} \label{eqn:150}
    E(q) = \frac12 \int_{\R^n} |\del_x q|^2 + (|q|^2 - 1)^2 \dd x
\end{equation}
is conserved. The Gross-Pitaevskii equation (GP) can be interpreted as the Hamiltonian evolutionary equation associated with this energy.
The well-posedness result in Zhidkov spaces was expanded to the cases $n = 2, 3$ by C. Gallo \cite{Gallo}.
Global-in-time well-posedness in the energy space $\{q \in H^1_{\loc}(\R^n): E(q) < \infty\}$ equipped with the metric
\begin{align*}
    d_E(q, p) = \|q - p\|_{Z^1(\R^n) + H^1(\R^n)} + \||q|^2 - |p|^2\|_{L^2(\R^n)}
\end{align*}
was obtained by P. G\'erard \cite{Gerard2006, Gerard2008} for $n=1, 2, 3$, and for $n = 4$ under smallness assumptions. Later R. Killip, T. Oh, O. Pocovnicu, and M. Vișan \cite{RomainTadahiroOanaMonica} established global-in-time well-posedness in the energy space for $n = 4$. More recently, the problem has been studied by P. Antonelli, L.E. Hientzsch, and P. Marcati \cite{Antonelli2023} in $n = 2, 3$ for general nonlinearities $f$ satisfying a Kato-type assumption. Local-in-time well-posedness is obtained in the energy space and extended globally in the defocusing case under some further assumptions. Regarding the case of non-finite energy, H. Pecher \cite{Pecher2012} established global-in-time well-posedness in three dimensions in $1 + H^s(\R^3)$ for $s \in (\frac56, 1)$.

We are concerned with the case $n = 1$. For $s \in \R$, we associate with solutions of (GP) the energy functionals
\begin{equation} \label{eqn:291}
    E^s(q) = \frac12 \big\|\del_x q\big\|_{H^{s-1}(\R)}^2 + \frac12 \big\||q|^2 - 1\big\|_{H^{s-1}(\R)}^2.
\end{equation}
Note that indeed $E^1 = E$. Our results are consequences of a pair of papers \cite{KochLiao, KochLiao2022} by
H. Koch and X. Liao, where for $s \geq 0$ they proved the global-in-time well-posedness of (GP) in the complete metric space
\begin{equation} \label{eqn:292}
    X^s = \{q \in H^s_{\loc}(\R): E^s(q) < \infty\} \, / \, \mathbb{S}^1 \,,
\end{equation}
equipped with the distance function
\begin{equation} \label{eqn:2}
    d^s(q, p) = \left( \int_{\R} \inf_{\lambda \in \mathbb{S}^1} \|\sech(y - \cdot) (\lambda q - p)\|_{H^s}^2 \dd y \right)^{\frac12} \,.
\end{equation}
We summarize several of their results, taken from \cite[Theorem 1.2, 1.3, Lemma 6.1]{KochLiao} and \cite[Theorem 1.5]{KochLiao2022}, in the following theorem.
\begin{restatable}[Global-in-time well-posedness of (GP) \cite{KochLiao, KochLiao2022}]{theorem}{} \label{thm:1}
    Let $s \geq 0$. The pair $(X^s, d^s)$ is a complete metric space, and the energy functional $E^s: X^s \longrightarrow \R$ is continuous. There exists a constant $C_0 > 0$ such that $d^s(1, q) \leq C_0 \sqrt{E^s(q)}$ for all $q \in X^s$.
    
    The Gross-Pitaevskii equation (GP) is globally-in-time well-posed in the metric space $(X^s, d^s)$ in the following sense: For any initial data $q_0 \in X^s$ there exists a unique global-in-time solution $q \in C(\R; X^s)$ of (GP) (see Definition \ref{def:1} below). For any $t \geq 0$ the Gross-Pitaevskii flow map $X^s \ni q_0 \mapsto q \in C([-t, t]; X^s)$ is continuous. There exists a constant $\tilde{C}_0(s, E^s(q_0))$ such that 
    \begin{equation} \label{eqn:155}
        \sup_{t \in \R} E^s(q(t)) \leq \tilde{C}_0 (s, E^s(q_0)) \, E^s(q_0) \,,
    \end{equation}
    and in the case $s \geq 1$ the energy $E(q(t))$, defined in (\ref{eqn:150}), is conserved.
\end{restatable}

\subsection{Related results for the hydrodynamic Gross-Pitaevskii equations}

The question of equivalence between Schr\"odinger equations and quantum hydrodynamical equations is relevant for the validity of classical approaches to quantum mechanics such as de Broglie-Bohm Theory \cite{Bohm} and stochastic mechanics \cite{Nelson}. It became a topic of controversy when Wallstrom raised some objections \cite{Wallstrom3, Wallstrom1, Wallstrom2}. We recommend \cite{Reddiger2023} for a review of these issues. The difficulty arises from possible vacuum regions, which complicate the definition of the inverse Madelung transform. This can be resolved via an additional ``Takabayasi's quantization condition'' \cite{Bianchini, Takabayasi}, which requires the winding numbers of the velocity field on closed loops to be quantized. Note that this condition trivially holds in one dimension. Wallstrom also raised objections regarding uniqueness \cite{Wallstrom2}. There are indeed non-uniqueness results for weak solutions to quantum hydrodynamical systems, which are related to a change in the number of non-vacuum connected components \cite{Markovich}.

Nevertheless, P. Antonelli and P. Marcati \cite{AntonelliMarcati, AntonelliMarcati2010} constructed weak solutions vanishing at infinity with finite but arbitrarily large energy, meaning that vacuum may appear, in $n = 2, 3$ for a general quantum hydrodynamical system. They use a polar decomposition technique in order to define the velocity field in the vacuum regions. In collaboration with H. Zheng, they extended this to $n = 1$ via a purely hydrodynamical approach \cite{AntonelliMarcatiZheng2019, Antonelli2021}. An alternative approach to finite energy weak solutions is explored in \cite{Antonelli2019, AntonelliMarcatiScandone}.

The well-posedness of the Euler-Korteweg system, a generalization of the compressible Euler equations which includes capillarity effects and contains (hGP) as a special case, was studied in higher dimensions by C. Audiard and B. Haspot \cite{Audiard2021, AudiardHaspot}. 
Similar to the approach we take is a paper by C. Audiard \cite{Audiard}, in which global-in-time well-posedness of (hGP) under smallness assumptions is shown in certain spaces for $n \geq 2$ by applying the Madelung transform to solutions to (GP). While they used scattering results to bound the solution away from $0$, we use a rather elementary argument that leads us to the aforementioned energy bound $E < \frac43$. 

A closely related paper is a work by H. Mohamad \cite{Mohamad}, which in \cite[Proposition 1.1]{Mohamad} states similar relations to our Theorem \ref{thm:3}, and then uses
a well-posedness result for (GP) to obtain a well-posedness result for (hGP) in $(1 + H^{k+1}) \times H^k$ with $k \in \N_{\geq 0}$ up to the appearance of vacuum. We discuss the similarities and differences in Remark \ref{rmk:3}.

\subsection{Functional analytic framework}

Our goal is to show a novel global-in-time well-posedness result for (hGP) with $(\rho, v) \in (1 + H^s) \times H^{s-1}$ where $s \geq 1$ and $n = 1$.
We achieve this under the assumptions $s \geq 1$ and $E < \frac43$ by passing the well-posedness result for (GP) in Theorem \ref{thm:1} through the Madelung transform (\ref{eqn:1}). The first assumption $s \geq 1$ ensures sufficient regularity for the energy $E$ to be defined, and for (hGP) to be interpretable in the sense of distributions. As an example, consider that $s \geq 1$ implies $v \in L^2(\R)$, and so the problematic square of a distribution $v^2$ appearing in (hGP)\textsubscript{2} does indeed exist. The second assumption $E < \frac43$ can also be understood as a ``regularity'' assumption: solutions below the critical energy of $\frac43$ can not have vacuum, that is points or intervals where $|q| = \sqrt{\rho} = 0$ (see Corollary \ref{cor:2} below). As a result, singularities are avoided in the hydrodynamic formulation. Due to conservation of energy, the absence of vacuum is guaranteed for all times. Note that this energy assumption is sharp in the sense that the black soliton solution $q(t,x) = \tanh(x)$ to (GP) has a zero $\tanh(0) = 0$, while also having energy $E(\tanh) = \frac43$. 

We collect now some results that ensure the absence of vacuum, given certain energy bounds. We start with the following lemma.

\noindent
\begin{minipage}[c]{0.7\linewidth}
\begin{lemma} \label{lem:8}
    Consider the function $\tilde{b}: [0, 1] \longrightarrow [0, \frac43]$ defined by
    \begin{equation*}
        \tilde{b}(\delta) = \frac43 - 2 \delta + \frac23 \delta^3.
    \end{equation*}
    This is a strictly decreasing bijection (see Fig. 1) whose inverse we denote by \\ $\tilde{\delta}(b): [0,\frac43] \longrightarrow [0, 1]$. We have
    \begin{align*}\setcounter{tocdepth}{2}
        \tilde{b}(\delta) &= \min \big\{ E(q): q \in H^1_{\loc}(\R), \, \inf_{x \in \R} |q(x)| \leq \delta \big\} \,, \\
        \tilde{\delta}(b) &= \min \big\{ \inf_{x \in \R} |q(x)| : q \in H^1_{\loc}(\R), E(q) \leq b \big\} \,.
    \end{align*}
\end{lemma}
\end{minipage} 
\begin{minipage}[c]{0.3\linewidth}
\captionsetup{labelformat=custom, labelsep=custom}
\begin{figure}[H]
    \centering
    \vspace{-0.6cm}
    \begin{tikzpicture}[scale=1.5]
      \draw[->] (-0.07, 0) -- (1.5, 0) node[right] {$\delta$};
      \draw[->] (0, -0.07) -- (0, 1.7) node[above] {$b$};
      \draw[-] (-0.07, 1.333) -- (0.07, 1.333) node[label={[label distance=0.5mm]180:{$\frac43$}}] {};
      \draw[-] (1, -0.07) -- (1, 0.07) node[label={[label distance=0.5mm]270:{$1$}}] {};
      \draw[-] (0, 0) -- (0, 0) node[label={[label distance=-1.0mm]250:{$0$}}] {};
      \draw[scale=1, domain=0:1, smooth, variable=\x, blue] plot ({\x}, {1.333 - 2*\x + 0.6666*\x*\x*\x});
    \end{tikzpicture}
    \vspace{-0.45cm}
    \captionsetup{width=1\linewidth}
    \caption{{ Graph of $\tilde{b}$}}
\end{figure}
\end{minipage}
This lemma is a stronger version of \cite[Lemma 1]{BethuelGravejatSaut}. The proof of a slightly more general Lemma \ref{lem:0} is given in the appendix.
As a consequence of Lemma \ref{lem:8}, the ``energy gap'' $\frac43 - E(q)$ yields an explicit lower bound for the distance of $|q|$ to zero. Due to conservation of the energy $E(q)$, we obtain the following corollary.
\begin{corollary} \label{cor:2}
    For any solution $q \in C(\R; X^1)$ of (GP) (see Definition \ref{def:1}), we have
    \begin{equation} \label{eqn:158}
        E(q_0) < b < \frac43 \Longrightarrow \inf_{(t, x) \in \R^2} |q(t, x)| > \tilde{\delta}(b) > 0\,.
    \end{equation}
\end{corollary}

We thus consider solutions $q$ of (GP) in $X^s$, $s \geq 1$ with energy
\begin{equation} \label{eqn:3}
    E(q) < \frac43 \,,
\end{equation}
recalling the definitions (\ref{eqn:150}) - (\ref{eqn:2}) of $X^s$ and $E$.
We look for solutions $(\rho, v)$ of (hGP) in the function space
\begin{equation} \label{eqn:4}
    \mathcal{Y}^s = (1 + H^s(\R;\R)) \times H^{s-1}(\R;\R) \,,
\end{equation}
equipped with the metric
\begin{equation} \label{eqn:5}
    \theta^s((\rho, v), (\eta, w)) = \|\rho - \eta\|_{H^s} + \|v - w\|_{H^{s-1}} \,.
\end{equation}
We define the analogous energy
\begin{equation}
    \mathcal{E}(\rho, v) = E(\mathcal{M}^{-1}(\rho, v)) = \frac12 \int_{\R} \frac{(\del_x \rho)^2}{4 \rho} + \rho v^2 + (\rho - 1)^2 \dd x.
\end{equation}
Here the inverse Madelung transform is defined as
\begin{equation} \label{eqn:201}
    \mathcal{M}^{-1}(\rho, v) (x) = \big(\sqrt{\rho(x)} e^{i \phi(x)}\big) \, \mathbb{S}^1 = \big\{\lambda \sqrt{\rho(x)} e^{i \phi(x)}: \lambda \in \mathbb{S}^1\big\} \,,
\end{equation}
where $\phi$ is any spatial primitive of $v$, i.e. $\del_x \phi = v$. Note that the energy $E$ is indeed well-defined on equivalence classes under multiplication by $\mathbb{S}^1$, and furthermore that the space $X^s$ consists of such equivalence classes, and is hence a suitable domain for the Madelung transform $\mathcal{M}$, given in (\ref{eqn:1}).

In order to transform solutions of (GP) into solutions of (hGP) via the Madelung transform, we establish an equivalence between the relevant function spaces $(X^s, d^s)$ and $(\mathcal{Y}^s, \theta^s)$. Specifically, we prove a local bilipschitz equivalence between the distance functions $d^s$ and $\theta^s$ for all $s > \frac12$. While our main result only holds for $s \geq 1$, our approach has the potential to be extended to the case $\frac12 < s < 1$ if one finds a way to make sense of (hGP)\textsubscript{2} in such a low regularity setting. For example, this may be possible via a local smoothing result, as in \cite{KillipVisan} (see Remark \ref{rmk:4} below). When $\frac12 < s < 1$, the absence of vacuum can still be ensured by a smallness assumption of the form
\begin{equation*} \label{eqn:6}
    E^\mu(q) < \epsilon_0(\mu) << 1 \,,
\end{equation*}
where $\mu > \frac12$ (see (\ref{eqn:156}) below). This smallness condition can also replace $E < \frac43$ in the case $s \geq 1$, $\mu \leq s$.
Specifically, we have the following Lemma \ref{lem:8.1} as a replacement for Lemma \ref{lem:8}.
\begin{lemma} \label{lem:8.1}
    For $\delta \in [0, 1]$ and $\mu > \frac12$ define
    \begin{equation*}
        E^\mu_\delta = \inf\big\{E^\mu(q): q \in H^\mu_{\loc}, \inf_{x \in \R} |q(x)| \leq \delta\big\} \,.
    \end{equation*}
    Then $E^\mu_1 = 0$, the function $\delta \mapsto E^\mu_\delta$ is decreasing, and there exists a constant $\tilde{C}(\mu) > 0$ so that
    \begin{equation} \label{eqn:26}
        E^\mu_\delta \geq \frac{(1 - \delta)^2}{\tilde{C}(\mu)}\,.
    \end{equation}
\end{lemma}
This Lemma is also a special case of Lemma \ref{lem:0}.
By (\ref{eqn:155}) there exists for any $\mu > \frac12$ a constant $c(\mu) > 0$ such that \begin{equation} 
\label{eqn:7}
    E^\mu(q_0) < \epsilon \Longrightarrow \sup_{t \in \R} E^\mu(q(t)) < c(\mu) \, \epsilon
\end{equation}
for all $\epsilon \in (0, 1)$ and any solution $q \in C(\R; X^\mu)$ of (GP).
Not attempting to obtain a sharp bound, we state the analogous of Corollary \ref{cor:2}.
\begin{corollary} \label{cor:3}
    Let $\mu > \frac12$ and define
    \begin{equation} \label{eqn:156}
        \epsilon_0(\mu) = \max\left\{ \frac12, \frac{1}{4 c(\mu) \tilde{C}(\mu)} \right\} \,.
    \end{equation}
    For any solution $q \in C(\R; X^\mu)$ of (GP) (see Definition \ref{def:1}), we have
    \begin{equation} \label{eqn:157}
        E^\mu(q_0) < \epsilon < \epsilon_0(\mu) \Longrightarrow \inf_{(t, x) \in \R^2} |q(t, x)| > 1 - \sqrt{\epsilon} \sqrt{c(\mu) \tilde{C}(\mu)} > \frac12\,.
    \end{equation}
\end{corollary}
\begin{proof}
    We prove the contrapositive. Suppose $\inf_{(t, x) \in \R^2} |q(t, x)| \leq \delta \coloneqq 1 - \sqrt{\epsilon} \sqrt{c(\mu) \tilde{C}(\mu)}$ and note that $\delta \in (0, 1)$. Using the definition of $E^\mu_{\tilde{\delta}}$ and (\ref{eqn:26}), this implies that for any $\tilde{\delta} > \delta$ there exists $t \in \R$ with 
    \begin{equation*}
        E^\mu(q(t)) \geq E^\mu_{\tilde{\delta}} \geq \frac{(1 - \tilde{\delta})^2}{\tilde{C}(\mu)} \,.
    \end{equation*}
    In particular
    \begin{equation*}
        \sup_{t \in \R} E^\mu(q(t)) \geq \frac{(1 - \delta)^2}{\tilde{C}(\mu)} = c(\mu) \epsilon \,,
    \end{equation*}
    so (\ref{eqn:7}) implies $E^\mu(q_0) \geq \epsilon$.
\end{proof}
As the energies $E^\mu$ still provide a lower bound for the distance of $|q|$ to zero, we can use the smallness assumption $E^\mu < \epsilon_0(\mu)$ as a substitute for $E < \frac43$.
We define for $\mu > \frac12$ the energies
\begin{equation}
    \mathcal{E}^\mu(\rho, v) = E^\mu(\mathcal{M}^{-1}(\rho, v)) \,.
\end{equation}

\subsection{Main results}

For both the Gross-Pitaevskii equation (GP) and its hydrodynamic formulation (hGP), there are three key objects in our function framework: the energy, the space and the metric. 
We summarize the definitions given in \S 1.2 in the following diagram:
\begin{equation} \label{eqn:12}
    \arraycolsep=1.4pt\def\arraystretch{1.4}
    \begin{array}{|rl|}
        \hline
        E^s(q) &= \frac12 \| \del_x q\|_{H^{s-1}}^2 + \frac12 \||q|^2 - 1\|_{H^{s-1}}^2 \\
        X^s &= \{q \in H^s_{\loc}(\R; \C): E^s(q) < \infty\} \, / \, \mathbb{S}^1 \\
        d^s(q, p) &= \left(\int_{\R} \inf_{\lambda \in \mathbb{S}^1} \|\sech(y - \cdot) (\lambda q - p)\|_{H^s}^2 \dd y \right)^{\frac12}
        \\ & \\
        \mathcal{M} \Bigg\downarrow
        \quad &
        \begin{matrix} 
            q &= \sqrt{\rho} e^{i \phi}, \, \rho = |q|^2, \, v = \del_x \phi \\
            p &= \sqrt{\eta} e^{i \psi}, \, \eta = |p|^2, \, w = \del_x \psi
        \end{matrix}
        \quad
        \Bigg\uparrow \mathcal{M}^{-1}
        \\ & \\
        \mathcal{E}^s(\rho, v) &= E^s(\mathcal{M}^{-1}(\rho, v)) \\
        \mathcal{Y}^s &= \{(\rho, v) \in (1 + H^s(\R; \R)) \times H^{s-1}(\R; \R) \} \\
        \;\, \theta^s((\rho, v), (\eta, w)) &= \|\rho - \eta\|_{H^s} + \|v - w\|_{H^{s-1}} \\
        \hline
    \end{array}
\end{equation}
Here the Madelung transform and its inverse
\begin{equation*}
    \mathcal{M}(q) = \Big(|q|^2, \Imag\Big[\frac{\del_x q}{q}\Big]\Big) \qquad \qquad
    \mathcal{M}^{-1}(\rho, v) = \big(\sqrt{\rho} e^{i \phi}\big) \, \mathbb{S}^1, \quad \del_x \phi = v
\end{equation*}
are given in (\ref{eqn:1}) and (\ref{eqn:201}) respectively.
Recall also the explicit forms of the energies $E$ and $\mathcal{E}$ in the most important case $s = 1$:
\begin{equation*}
    E(q) = \frac12 \int_{\R} |\del_x q|^2 + (|q|^2 - 1)^2 \dd x \qquad \qquad \mathcal{E}(\rho, v) = \frac12 \int_{\R} \frac{(\del_x \rho)^2}{4 \rho} + \rho v^2 + (\rho - 1)^2 \dd x \,.
\end{equation*}
Our first main result is the following theorem, which is central to our strategy as it establishes a local bilipschitz equivalence between the metrics $d^s$ and $\theta^s$. We require $s > \frac12$ to use $L^\infty$ embeddings and certain product estimates.

\begin{restatable}[Local bilipschitz equivalence of $d^s$ and $\theta^s$]{theorem}{equivalence} \label{thm:3}
    Let $s > \frac12$ and $r, \delta > 0$. Consider measurable functions $\rho, \eta, \phi, \psi: \R \longrightarrow \R$ so that $q, p \in \mathcal{S}'(\R) \cap H_{\loc}^s(\R)$ and $|q|, |p| > \delta$, where $q = \sqrt{\rho} e^{i \phi}$ and $p = \sqrt{\eta} e^{i \psi}.$ There exist constants $C_1(s, \delta, r), C_2(s, r) > 0$ so that the following hold:
    \begin{enumerate}[(i)]
        \item If $d^s(1, q), d^s(1, p) < r$, then
            \begin{equation*}
                \theta^s((\rho, \del_x \phi), (\eta, \del_x \psi)) \leq C_1(s, \delta, r) \, d^s(q, p) \,.
            \end{equation*}
        \item If $\theta^s((1,0), (\rho, \del_x \phi)), \theta^s((1,0), (\eta, \del_x \psi)) < r$, then
        \begin{equation*}
            d^s(q, p) \leq C_2(s, r) \, \theta^s((\rho, \del_x \phi), (\eta, \del_x \psi)) \,.
        \end{equation*}
    \end{enumerate}
\end{restatable}
\begin{corollary} \label{cor:1}
    Let $s \geq 1$, $\frac12 < \mu < 1$. For all $b < \frac43$ and $\epsilon < \epsilon_0(\mu)$ the maps
    \begin{equation*}
        \big(\big\{q \in X^s: E(q) < b \big\}, d^s\big) \xrightarrow{\quad \mathcal{M} \quad} \big(\big\{(\rho, v) \in \mathcal{Y}^s: \mathcal{E}(\rho,v) < b  \big\}, \theta^s\big)
    \end{equation*}
    and
    \begin{equation*}
        \big(\big\{q \in X^s: E^\mu(q) < \epsilon \big\}, d^s\big) \xrightarrow{\quad \mathcal{M} \quad} \big(\big\{(\rho, v) \in \mathcal{Y}^s: \mathcal{E}^\mu(\rho, v) < \epsilon \big\}, \theta^s\big)
    \end{equation*}
    are bilipschitz equivalences. Recall that a bilipschitz equivalence is a map which is bijective, Lipschitz continuous, and has a Lipschitz continuous inverse. Here $\epsilon_0(\mu)$ is a constant defined in (\ref{eqn:156}).
\end{corollary}

Our second main result is the global-in-time well-posedness of the hydrodynamic Gross-Pitaevskii equations.

\begin{definition}[Solution to (hGP)] \label{def:2}
    Let $s \geq 1$ and $0 \in I \subset \R$ be an open time interval or the real line, and let $(\rho_0, v_0) \in \mathcal{Y}^s$ with $\rho_0 > 0$. A solution to (hGP) with initial data $(\rho_0, v_0)$ is a pair $(\rho, v) \in C(I; \mathcal{Y}^s)$ with $\rho > 0$ which solves (hGP) in the sense of distributions and fulfills $(\rho, v)(0) = (\rho_0, v_0)$.
\end{definition}
\begin{restatable}[Global-in-time well-posedness of (hGP) for $s \geq 1$]{theorem}{gwphgp} \label{thm:2}
    Let $s \geq 1$. The hydrodynamic Gross-Pitaevskii equations (hGP) are globally-in-time well-posed in the metric space $(\mathcal{Y}^s, \theta^s)$ for initial data $(\rho_0, v_0) \in \mathcal{Y}^s$ with $\mathcal{E}(\rho_0, v_0) < \frac43$ in the following sense:
    
    There exists a solution $(\rho, v) \in C_b(\R; \mathcal{Y}^s)$ to (hGP) (see Definition \ref{def:2}). It is the unique solution that fulfills 
    \begin{equation} \label{eqn:151}
        \mathcal{E}(\rho(t), v(t)) = \mathcal{E}(\rho_0, v_0) < \frac43
    \end{equation}
    for all $t \in \R$. For any $T \geq 0$ the solution map
    \begin{align*}
        \Big\{(\rho_0, v_0) \in \mathcal{Y}^s: \mathcal{E}(\rho_0, v_0) < \frac43 \Big\} &\longrightarrow C_b([-T, T]; \mathcal{Y}^s) \\
        (\rho_0, v_0) &\longmapsto (\rho, v)
    \end{align*}
    is continuous.
    
    For all $\frac12 < \mu < 1$ there exist constants $c(\mu), \epsilon_0(\mu) > 0$, defined in (\ref{eqn:7}) and (\ref{eqn:156}), so that if we replace the assumption $\mathcal{E}(\rho_0, v_0) < \frac43$ by $\mathcal{E}^\mu(\rho_0, v_0) < \epsilon < \epsilon_0(\mu)$, then the above statement holds with (\ref{eqn:151}) replaced by $\mathcal{E}^\mu(\rho(t), v(t)) < c(\mu) \, \epsilon$.
\end{restatable}

\begin{remark} \label{rmk:3}
    In \cite[Theorem 1.2]{Mohamad} the well-posedness of (hGP) until the appearance of vacuum is shown in $(1 + H^k) \times H^{k-1}$ for $k \in \N_{\geq 1}$. Furthermore, continuity properties for the Madelung transform between $(\mathcal{Y}^k, \theta^k)$ and the space
    \begin{equation*}
        \tilde{E}^k = \{q \in L^\infty(\R; \C): 1 - |q|^2 \in L^2(\R), \del_x q \in H^{k-1}(\R)\}
    \end{equation*}
    are established.
    Specifically, a strong metric
    \begin{equation*}
        \tilde{d}^k(q, p) = \|q - p\|_{L^\infty(\R)} + \||q|^2 - |p|^2\|_{L^2(\R)} + \|\del_x q - \del_x p\|_{H^{k-1}(\R)} 
    \end{equation*}
    and a weak metric
    \begin{equation*}
        \tilde{d}_{\loc}^k(q, p) = \|q - p\|_{L^\infty([-1, 1])} + \||q|^2 - |p|^2\|_{L^2(\R)} + \|\del_x q - \del_x p\|_{H^{k-1}(\R)}\
    \end{equation*}
    are considered, and the following is shown \cite[Proposition 1.1]{Mohamad}:
    \begin{center}
        \begin{tikzpicture}
            \node (A) at (0,0) {$(\mathcal{Y}^k \cap \{\rho > 0\}, \theta^k)$};
            \node (B) at (-6,0.5) {$(\tilde{E}^k \cap \{|q| > 0\}, \tilde{d}^k)$};
            \node (C) at (-6,-0.5) {$(\tilde{E}^k \cap \{|q| > 0\}, \tilde{d}_{\loc}^k)$};
            \node (D) at (6,0.5) {$(\tilde{E}^k \cap \{|q| > 0\}, \tilde{d}^k)$};
            \node (E) at (6,-0.5) {$(\tilde{E}^k \cap \{|q| > 0\}, \tilde{d}_{\loc}^k)$};
            \draw[->] (B) -- node[above]{loc. Lipschitz} (A);
            \draw[->] (C) -- node[below]{not loc. Lipschitz} (A);
            \draw[->] (A) -- node[above]{not cont.} (D);
            \draw[->] (A) -- node[below]{cont.} (E);
        \end{tikzpicture}
    \end{center}
    There are several ways in which Theorem \ref{thm:3} improves upon \cite[Theorem 1.2]{Mohamad}. The first is that our result covers the fractional cases as well. The second is that we obtain a full local bilipschitz equivalence
    \begin{center}
        \begin{tikzpicture}
            \node (A) at (0,0) {$(\mathcal{Y}^s \cap \{\rho > 0\}, \theta^s)$};
            \node (B) at (-6,0) {$(X^s \cap \{|q| > 0\}, d^s)$};
            \node (D) at (6,0) {$(X^s \cap \{|q| > 0\}, d^s)$};
            \draw[->] (B) -- node[above]{loc. Lipschitz} (A);
            \draw[->] (A) -- node[above]{loc. Lipschitz} (D);
        \end{tikzpicture}
    \end{center}
    and do so down to $s > \frac12$, while they have to work with two different metrics on the side of (GP) for the two directions of estimates. Furthermore, they state well-posedness up to the appearance of vacuum, while we use energy bounds to ensure this absence of vacuum for all times. Lastly, we find solutions which are uniformly bounded.
\end{remark}

\begin{remark} \label{rmk:1}
    Previously, P.E. Zhidkov \cite[Theorem III.3.1]{Zhidkov2001} studied the stability of solutions in the Zhidkov space $Z^1(\R)$ (see (\ref{eqn:202})) near space-homogeneous solutions $\Phi$, such as the constant solution $ \Phi = 1$, with respect to the distance $\theta^1$. For the case $s = 1$ he derived similar estimates as above under smallness assumptions, although he did not formulate a well-posedness result. Curiously, in \cite[Cor. III.3.5]{Zhidkov2001} he proved furthermore that for any ball $B \subset \R$, if the initial $\theta^1$-distance between the perturbed and the space-homogeneous solution is small, then for all times also the distance
    \begin{equation*}
        \inf_{\lambda \in \mathbb{S}^1} \|\lambda q - \Phi\|_{W^{1,2}(B)}
    \end{equation*}
    is small. This can be interpreted as a weaker form of the estimate $d^1 \lesssim \theta^1$ we derive (see Lemma  \ref{lem:36} and Remark \ref{rmk:2} below).
\end{remark}

\begin{remark} \label{rmk:4}
    As both Theorem \ref{thm:1} and Theorem \ref{thm:3} work for all $s > \frac12$, it may be possible to extend Theorem \ref{thm:2} to the case $\frac12 < s < 1$. The problem is that for $v \in H^{s-1} \not\subseteq L^2$ the product of distributions $v^2 = v \cdot v$ is not necessarily defined. Nevertheless, it may be possible to find global distributional solutions. For example, in the paper \cite{KillipVisan} by R. Killip and M. Vișan global-in-time well-posedness of the KdV equation in $H^{-1}$ is first shown in the sense that the solution map $\R \times \mathcal{S} \longrightarrow \mathcal{S}$ extends to a continuous mapping $\R \times H^{-1} \longrightarrow H^{-1}$, and some other conditions are fulfilled. 
    In our case, it is similarly true that for any $T \geq 0$ the solution map
    \begin{equation*}
        \Big\{(\rho_0, v_0) \in \mathcal{Y}^1: \mathcal{E}^s(\rho_0, v_0) < \epsilon_0(s) \Big\} \longrightarrow C_b([-T,T]; \mathcal{Y}^1)
    \end{equation*}
    has a unique continuous extension to a map
    \begin{equation*}
        \Big\{(\rho_0, v_0) \in \mathcal{Y}^s: \mathcal{E}^s(\rho_0, v_0) < \epsilon_0(s) \Big\} \longrightarrow C_b([-T,T]; \mathcal{Y}^s) \,.
    \end{equation*}
    This extension is given by the conjugation of the corresponding solution map for (GP) at regularity $s$ with the Madelung transform.
    R. Killip and M. Vișan then furthermore show a local smoothing result, which implies that the solution map produces functions in $L^2_{\loc,t,x}$. As a result, the equation is indeed solved in the sense of distributions. We do not know if such a local smoothing result holds in our case.
\end{remark}

\textbf{Organization of the paper.}
In \S 2 we prove Theorem \ref{thm:3}, the local bilipschitz equivalence of $(X^s, d^s)$ and $(\mathcal{Y}^s, \theta^s)$. In \S 3 we prove Theorem \ref{thm:2}, the global-in-time well-posedness of the hydrodynamic Gross-Pitaevskii equations.

\textbf{Acknowledgements.} 
I would like to thank Sarah Hofbauer for her help with reviewing the literature, and Phillipe Gravejat for directing my attention to the highly relevant paper \cite{Mohamad} by H. Mohamad. I would also like to thank the anonymous reviewer for his detailed and valuable comments. I am especially thankful to my supervisor Xian Liao for proposing this problem and strategy, and for her patience during many hours of discussion.

\textbf{Statements and Declarations.} \\
\textit{Funding}: This work was funded by the Deutsche Forschungsgemeinschaft (DFG, German Research Foundation) – Project-ID 258734477 – SFB 1173. \\
\textit{Version}: This version of the article has been accepted for publication, after peer review but is not the Version of Record and does not reflect post-acceptance improvements, or any corrections. The Version of Record is available online at: https://doi.org/10.1007/s00033-023-02089-4
\section{Local bilipschitz equivalence of $(X^s, d^s)$ and $(\mathcal{Y}^s, \theta^s)$}

The goal of this section is to prove Theorem \ref{thm:3}.
In \S 2.1, we introduce the necessary notations, definitions, and basic results required for the rest of the paper. We split the proof of the two statements $(i)$ and $(ii)$ of Theorem \ref{thm:3} into \S 2.2 and \S 2.3.

\subsection{Notations and preliminaries}

We use the notations $\R_\geq = \{r \in \R: r \geq 0\}$ and $\R_+ = \{r \in \R: r > 0\}$. We write $C$ or $C(...)$ for various constants with possible dependence on other quantities. These may change from one line to the next.
We denote by $\mathcal{D}' = \mathcal{D}'(\R) = \mathcal{D}'(\R;\C)$ the space of distributions and by $\mathcal{S}' = \mathcal{S}'(\R) = \mathcal{S}'(\R; \C)$ the space of tempered distributions. In general, if for a family of function spaces, such as the $L^p$-spaces, we write just ``$L^p$'', then we mean $L^p(\R; \C)$.

We write $\widehat{f}$ for the Fourier transform
\begin{align*}
    \hat{f}(\xi) = \frac{1}{\sqrt{2 \pi}} \int_{\R} e^{- i x \xi} f(x) \dd x
\end{align*}
for a Schwartz function $f \in \mathcal{S}$ and extend the definition as usual to the tempered distributions $f \in \mathcal{S}'$. Let $s \in \R$. We define the Sobolev space
\begin{equation*}
    H^s = H^s(\R) = H^s(\R;\C) = \{f \in \mathcal{S}'(\R;\C): \|f\|_{H^s} < \infty\}
\end{equation*}
with norm
\begin{equation*}
    \|f\|_{H^s} = \| f \|_{H^s(\R)} = \big\| \langle \xi \rangle^s \widehat{f}\, \big\|_{L^2(\R)}\,.
\end{equation*}
Here $\langle \xi \rangle = \sqrt{1 + |\xi|^2}$.
We also define the quasinorm of the homogeneous Sobolev space
\begin{equation*}
    \|f\|_{\dot{H}^s} = \| f \|_{\dot{H}^s(\R)} = \big\| |\xi|^s \widehat{f}\, \big\|_{L^2(\R)}\,.
\end{equation*}

Let $p \in [1,\infty)$. For $s \geq 0$, let $\alpha \in [0, 1)$ and $m \in \Z$ so that $s = m + \alpha$. Let $B \subset \R$ be a non-empty open interval. We define the Sobolev-Slobodeckij space
\begin{equation*}
    W^{s,p}(B) = \{f \in \mathcal{D}'(B):  \|f\|_{W^{s,p}(B)} < \infty\}
\end{equation*}
with norm
\begin{equation*}
    \|f\|_{W^{s,p}(B)} = \left( \sum_{k = 0}^m \|\del^k f\|_{L^p(B)}^p + \int_B \int_B \frac{|\del^m f(x) - \del^m f(y)|^p}{|x - y|^{1 + \alpha p}} \dd x \dd y \right)^{\frac{1}{p}} \,.
\end{equation*}
We define $W_0^{s,p}(B) = \overline{\mathcal{D}(B)}^{W^{s,p}(B)}$. For $s < 0$ we define $W^{s,p'}(B) = (W_0^{-s,p}(B))^\ast$, where $\frac{1}{p} + \frac{1}{p'} = 1$. We refer the reader to the book \cite{McLean} by W. McLean for a comprehensive exposition. For the convenience of the reader, let us recall some well-known results on Sobolev spaces that may be used without mention.

\subsubsection{Fractional Sobolev spaces on $B$ and $\R$}

The only bounded domains we use are balls $B$, and on those we use the Sobolev-Slobodeckij spaces $W^{s,2}(B)$. On the whole real line $\R$ we use $H^s = H^s(\R)$.
\begin{lemma}[$W^{s,2}(B)$ and $H^s(\R)$] \label{lem:21}
    Let $s \in \R$ and $R > 0$. Let $\tilde{B} \subset B \subseteq \R$ be concentric balls of radius $\frac{R}{2}$ and $R$. Set $B_k = B + k R$ and $\tilde{B}_k = \tilde{B} + k R$ for $k \in \Z$.
    \begin{enumerate}[(i)]
        \item There exists a natural isomorphism $H^{-s} \cong (H^s)^\ast$ (see \cite[p. 76]{McLean}).
        
        \item $H^s = W^{s,2}(\R)$ and
        \begin{equation*}
            \|f\|_{W^{s,2}(B)} \leq C_1 \min\{\|F\|_{H^s}: F \big\vert_B = f\} \leq C_2 \|f\|_{W^{s,2}(B)}\,.
        \end{equation*}
        (see \cite[p. 77, (3.23) + Theorem 3.18, 3.19]{McLean}).
        
        \item For $s \geq 0$, there exists a bounded linear extension operator $E: W^{s,2}(B) \longrightarrow H^s$ with $E f \big\vert_B = f$ (see \cite[Theorem A.4]{McLean}).
        
        \item For $s \geq 0$, there exists a constant $C(s, R)$ so that
        \begin{equation*}
            \sum_{k \in \Z} \|f\|_{W^{s,2}(\tilde{B}_k)}^2 \leq \|f\|_{H^s}^2 \leq C(s, R) \sum_{k \in \Z} \|f\|_{W^{s,2}(B_k)}^2 \leq 4 C(s, R) \|f\|_{H^s}^2
        \end{equation*}
        and
        \begin{equation*}
            \|f\|_{H^{-s}}^2 \leq C(s, R) \sum_{k \in \Z} \|f\|_{W^{-s,2}(B_k)}^2.
        \end{equation*}
        
        \item If $s > \frac12$, then $\|f g\|_{H^s} \leq C(s) \|f\|_{H^s} \|g\|_{H^s}$ (see \cite[Cor. 2.87]{BCD}). By use of the extension operator, we also have $\|f g\|_{W^{s,2}(B)} \leq C(s) \|f\|_{W^{s,2}(B)} \|g\|_{W^{s,2}(B)}$.

    \end{enumerate}
\end{lemma}
\begin{proof}
    We only have to prove $(iv)$. We start with the first inequality in the sequence.
    The case $s = 0$ is trivial, so we assume $s > 0$. Here the statement is trivial for the terms with integer regularity, and for the fractional terms we estimate
    \begin{align*}
        \sum_{k \in \Z} \int_{\tilde{B}_k} \int_{\tilde{B}_k} \frac{|\del^m f(x) - \del^m f(y)|^2}{|x - y|^{1 + 2 \alpha}} \dd x \dd y &\leq \sum_{j, k \in \Z} \int_{\tilde{B}_j} \int_{\tilde{B}_k} \frac{|\del^m f(x) - \del^m f(y)|^2}{|x - y|^{1 + 2 \alpha}} \dd x \dd y \\
        &= \int_{\R} \int_{\R} \frac{|\del^m f(x) - \del^m f(y)|^2}{|x - y|^{1 + 2 \alpha}} \dd x \dd y \,.
    \end{align*}
    The third inequality in the sequence follows trivially.
    We show the second inequality first for $s > 0$. For the terms with integer regularity the statement is again trivial, so we focus on the fractional part\footnote{The proof in the published version of this paper contains a mistake, which is corrected in this version.}.
    Here
    \begin{align*}
        \int_{\R} \int_{\R} \frac{|\del^m f(x) - \del^m f(y)|^2}{|x - y|^{1 + 2 \alpha}} \dd x \dd y
        &= \sum_{j, k \in \Z} \int_{\tilde{B}_k} \int_{\tilde{B}_j} \frac{|\del^m f(x) - \del^m f(y)|^2}{|x - y|^{1 + 2 \alpha}} \dd x \dd y
        \\ &\leq \sum_{k \in \Z} \int_{B_{k-1} \cup B_{k+1}} \int_{B_{k-1} \cup B_{k+1}} \frac{|\del^m f(x) - \del^m f(y)|^2}{|x - y|^{1 + 2 \alpha}} \dd x \dd y
        \\ &+ \sum_{k \in \Z} \int_{\tilde{B}_k} 2 (|\del^m f(x)|^2 + |\del^m f(x)|^2) \underset{|k - j| \geq 2}{\sum_{j \in \Z}} \int_{\tilde{B}_j} \frac{1}{|x - y|^{1 + 2 \alpha}} \dd x \dd y
        \\ &= (I) + (II) \,.
    \end{align*}
    Clearly
    \begin{align*}
        (II) &\leq C(s, R) \sum_{k \in \Z} \|\del^m f\|_{L^2(B_k)}^2
    \end{align*}
    and
    \begin{align*}
        (I) &\leq C(s, R) \sum_{k \in \Z} \|\del^m f\|_{W^{\alpha,2}(B_{k-1} \cup B_{k+1})}^2 \leq C(s, R) \sum_{k \in \Z} \|f\|_{W^{s,2}(B_k)}^2 \,.
    \end{align*}
    The final estimate for $(I)$ uses a partition of unity with the covering $\bigcup_{j=-2}^2 B_{k+j}$ of $B_{k-1} \cup B_{k+1}$.
    Now we show the inequality for negative regularities.
    Let $f \in H^{-s}(\R)$, $g \in H^s(\R)$ and denote by $\langle f, g \rangle$ the dual pairing.
    We decompose $g = \sum_{k \in \Z} \eta_k g$, where $\eta_k$ is a smooth partition of unity with $\supp \eta_k \subset B_k$, $\sum_{k \in \Z} \eta_k = 1$ and $\eta_k(x) = \eta_k(x + k R)$. Since $\eta_k g \in W^{s,2}_0(B_k)$ with $\|\eta_k g\|_{W^{s,2}(B_k)} \leq C(s, R, \eta_0) \|g\|_{W^{s,2}(B_k)}$, we can estimate
    \begin{align*}
        |\langle f, g \rangle| &\leq \sum_{k \in \Z} |\langle f, \eta_k g \rangle|
        \leq \sum_{k \in \Z} \|f\|_{W^{-s,2}(B_k)} \|\eta_k g\|_{W^{s,2}(B_k)}
        \leq \left( \sum_{k \in \Z} \|f\|_{W^{-s,2}(B_k)}^2 \right)^{\frac12} \left( \sum_{k \in \Z} \|\eta_k g\|_{W^{s,2}(B_k)}^2 \right)^{\frac12}
        \\ &\leq C(s, R, \eta_0) \left( \sum_{k \in \Z} \|f\|_{W^{-s,2}(B_k)}^2 \right)^{\frac12} \|g\|_{H^s(\R)} \,.
    \end{align*}
\end{proof}
\subsubsection{Estimates in $H^s$}

The following lemma states two crucial estimates. Such kinds of product estimates are well-known in the literature, see for example \cite[Proposition 2.7]{DanchinLiao}.

\begin{lemma} \label{lem:2}
    Let $s > \frac12$ and $f, g \in \mathcal{S}'$. There exists a constant $C(s)$ so that
    \begin{equation} \label{eqn:13}
        \|f g\|_{H^s} \leq C(s) \|g\|_{H^s} \big(\|f\|_{L^\infty} + \|f'\|_{H^{s-1}}\big)
    \end{equation}
    and
    \begin{equation} \label{eqn:14}
        \|f g\|_{H^{s-1}} \leq C(s) \|g\|_{H^{s-1}} \big(\|f\|_{L^\infty} + \|f'\|_{H^{s-1}}\big) \,.
    \end{equation}
\end{lemma}
\begin{proof}
    See Appendix B.
\end{proof}

In this section we often write $f'$ for the spatial derivative $\del_x f$. Recall the definitions (\ref{eqn:12}). For notational convenience, we sometimes prefer to use the variables
\begin{equation*}
     A = \sqrt{\rho} \quad \text{ and } \quad B = \sqrt{\eta}\,.
\end{equation*}
These variables are equivalent for the sake of our estimates, by which we mean specifically Lemma \ref{lem:403}. In order to prove this, we state two estimates regarding the action of a smooth function on Sobolev spaces. They are a direct consequence of some results in \cite{BCD}.

\begin{lemma}[{\cite[Theorem 2.87, Corollary 2.91]{BCD}}] \label{lem:211}
    Let $s > \frac12$ and $F \in C^\infty(\R; \R)$ with $F'(0) = F(0) = 0$. Let $u, v \in H^s(\R; \R) \cap L^\infty(\R; \R)$. We have the estimates
    \begin{equation} \label{eqn:153}
        \|F \circ u\|_{H^s} \leq C(s, F', \|u\|_{L^\infty}) \|u\|_{H^s}
    \end{equation}
    and 
    \begin{equation} \label{eqn:154}
        \|F \circ u - F \circ v\|_{H^s} \leq C(s, F'', \|u\|_{H^s}, \|v\|_{H^s}) \|u - v\|_{H^s} \,.
    \end{equation}
    An analysis of the proof in \cite{BCD} reveals that, more precisely, the constants depend on $\|F'\|_{C^{\lceil s \rceil + 1}(B_{\|u\|_{L^\infty}})}$ and $\|F''\|_{C^{\lceil s \rceil + 1}(B_{\|u\|_{L^\infty}})}$ respectively.
\end{lemma}

\begin{lemma} \label{lem:403}
    Let $s > \frac12$ and $\rho, \eta \in \mathcal{S}'(\R; \R) \cap H^s_{\loc}(\R; \R)$ with $\rho, \eta > 0$. Define $A = \sqrt{\rho}$ and $B = \sqrt{\eta}$. We have the estimates
    \begin{equation*}
        \|\rho - \eta\|_{H^s} \leq C_1 \big(s, \|A - 1\|_{H^s}, \|B - 1\|_{H^s}\big)\, \|A - B\|_{H^s}
    \end{equation*}
    and
    \begin{equation*}
        \|A - B\|_{H^s} \leq C_2 \big(s, \|\rho - 1\|_{H^s}, \|\eta - 1\|_{H^s}\big)\, \|\rho - \eta\|_{H^s} \,.
    \end{equation*}
\end{lemma}
\begin{proof}
    We apply Lemma \ref{lem:211} with $F(u) = u^2$ and obtain
    \begin{align*}
        \|A^2 - B^2\|_{H^s} &\leq \|(A-1)^2 - (B-1)^2\|_{H^s} + 2 \|A - B\|_{H^s} \\
        &\leq C\big(s, \|A - 1\|_{H^s}, \|B - 1\|_{H^s}\big) \, \|A - B\|_{H^s}\,.
    \end{align*}
    Similarly with any function $F \in C^\infty(\R; \R)$ that fulfills $F(u) = \sqrt{u + 1} - \frac12 u - 1$ for $x \geq 0$, we obtain
    \begin{align*}
        \|\sqrt{\rho} - \sqrt{\eta}\|_{H^s} &\leq \left\|\left(\sqrt{(\rho-1) + 1} - \frac12 (\rho - 1)\right) - \left(\sqrt{(\eta-1) + 1} - \frac12 (\eta - 1)\right)\right\|_{H^s} \\
        & \qquad \; \,+ \frac12 \|\rho - \eta\|_{H^s} \\
        &\leq C\big(s, \|\rho - 1\|_{H^s}, \|\eta - 1\|_{H^s}\big) \, \|\rho - \eta\|_{H^s}\,.
    \end{align*}
\end{proof}
The following lemma is also a consequence of Lemma \ref{lem:211} and will be used frequently in the subsequent section.
\begin{lemma} \label{lem:212}
    Let $s, \delta, R > 0$. There exists $C(s, \delta, R) > 0$ such that for any ball $B_0 \subset \R$ of radius $R$ and all $u \in W^{s,2}(B_0)$ with $|u| > \delta > 0$ we have
    \begin{equation*}
        \|u^{-1}\|_{W^{s,2}(B_0)} \leq C(s, \delta) \|u\|_{W^{s,2}(B_0)}\,.
    \end{equation*}
\end{lemma}
\begin{proof}
    This follows by applying Lemma \ref{lem:211} with any function $F \in C^\infty(\R; \R)$ so that $F(0) = F'(0) = 0$ and $F(x) = x^{-1}$ for $|x| > \frac{\delta}{2}$, and using the existence of an extension operator from Lemma \ref{lem:21} $(iii)$. Note that Lemma \ref{lem:211} requires real-valued functions, so we apply it to the real and imaginary parts of $u^{-1}$ separately.
\end{proof}

\subsection{Proof of Theorem \ref{thm:3} (i)}

Recall the definitions (\ref{eqn:12}), in particular $q = \sqrt{\rho} e^{i \phi}$ and $p = \sqrt{\eta} e^{i \psi}$, as well as $A = \sqrt{\rho}$ and $B = \sqrt{\eta}$.
We assume $s > \frac12$, $d^s(1, q), d^s(1, p) < r$ and $|q|, |p| > \delta > 0$. We have to prove that
\begin{equation*}
    \theta^s((\rho, \del_x \phi), (\eta, \del_x \psi)) \leq C (s, \delta, r) \, d^s(q, p) \,.
\end{equation*}
We do this by showing an estimate of the form
\begin{equation} \label{eqn:21}
    \theta^s \lesssim \sum_{k \in \Z} d^s_\ast \big\vert_{B_k} \lesssim d^s\,.
\end{equation}
Let us elaborate on the quantity in the middle before we start the proof.
Given a ball $B \subset \R$, we define for convenience the following notations:
\begin{align} \label{eqn:19}
    d_\ast^s\big\vert_B(q, p) &= \inf_{\lambda \in \mathbb{S}^1} \|\lambda q - p\|_{W^{s,2}(B)} \,, \\ \label{eqn:22}
    d^s\big\vert_B(q, p) &= \left( \int_{\R} \inf_{\lambda \in \mathbb{S}^1} \|\sech(y - \cdot) (\lambda q - p)\|_{W^{s,2}(B)}^2 \dd y \right)^{\frac12} \,.
\end{align}
\begin{lemma} \label{lem:36}
    Let $s > \frac12$ and let $B_0 = \{x \in \R: |x| < R\}$ be an open ball of radius $R > 0$ with center $0$. There exists $C(s, R) > 0$ so that
    \begin{equation} \label{eqn:203}
        d_\ast^s\big\vert_{B_0}(q, p) \leq C(s, R) \, d^s\big\vert_{B_0}(q, p)
    \end{equation}
    for all $q, p \in \mathcal{S}' \cap H^s_{\loc}$. As a consequence, for families of balls $B_k = B_0 + k R$ with $k \in \Z$ we have
    \begin{equation} \label{eqn:204}
        \sum_{k \in \Z} d_\ast^s\big\vert_{B_k}(q, p)^2 \leq C(s, R) \, d^s(q, p)^2 \,.
    \end{equation}
\end{lemma}
\begin{proof}
    As $\{y - x: x, y \in B_0\} \subseteq \{x \in \R: |x| < 2R\}$, there exists a finite constant $C(s, R) > 0$ such that $\sup_{y \in B_0} \|\sech(y - \cdot)^{-1}\|_{W^{s,2}(B_0)}^2 \leq C(s, R)$. The first estimate follows:
    \begin{align*}
        \inf_{\lambda \in \mathbb{S}^1} \|\lambda q - p\|_{W^{s,2}(B_0)}^2 &\leq C(s, R) \inf_{y \in B_0}  \inf_{\lambda \in \mathbb{S}^1} \|\sech(y - \cdot) (\lambda q - p)\|_{W^{s,2}(B_0)}^2 \\
        &\leq C(s, R) \int_{\R} \inf_{\lambda \in \mathbb{S}^1} \|\sech(y - \cdot) (\lambda q - p)\|_{W^{s,2}(B_0)}^2 \dd y \,.
    \end{align*}
    Using this and Lemma \ref{lem:21} $(iv)$, we obtain the second estimate:
    \begin{align*}
        \sum_{k \in \Z} d_\ast^s\big\vert_{B_k}(q, p)^2 &\leq C(s, R) \sum_{k \in \Z} d^s\big\vert_{B_k}(q, p)^2 \\
        &\leq C(s, R) \int_{\R} \inf_{\lambda \in \mathbb{S}^1} \sum_{k \in \Z} \|\sech(y - \cdot) (\lambda q - p)\|_{W^{s,2}(B_k)}^2 \dd y \\
        &\leq C(s, R) \, d^s(q, p)^2 \,.
    \end{align*}
\end{proof}

\begin{proof}[Proof of Theorem \ref{thm:3} (i)]
    Let $B_0 = \{x \in \R: |x| < 1\}$ and observe that
    \begin{align*}
        \allowdisplaybreaks
        &\quad\, \||q|^2 - |p|^2\|_{W^{s,2}(B_0)} \\ \nonumber
        &= \inf_{\lambda, \nu \in \mathbb{S}^1} \||\lambda q|^2 - |\nu p|^2\|_{W^{s,2}(B_0)} \\ \nonumber
        &= \inf_{\lambda, \mu, \nu\in \mathbb{S}^1} \||\lambda q - \mu|^2 - |\nu p - \mu|^2 + 2 (\Real(\lambda \overline{\mu} q) - \Real(\overline{\mu} \nu p))\|_{W^{s,2}(B_0)} \,. \\ \nonumber
    \end{align*}
    We can estimate
    \begin{align} \label{eqn:15}
        &\quad\, \||q|^2 - |p|^2\|_{W^{s,2}(B_0)} \\ \nonumber
        &\leq \inf_{\lambda, \mu, \nu \in \mathbb{S}^1} \Big\| \Real\Big( \big((\lambda q - \mu) - (\nu p - \mu)\big) \big(\overline{(\lambda q - \mu) + (\nu p - \mu)}\big) \Big) \Big\|_{W^{s,2}(B_0)} \\ \nonumber
        & \hspace{3em} + 2 \|\Real(\lambda \overline{\mu} q - \overline{\mu} \nu p)\|_{W^{s,2}(B_0)} \\ \nonumber
        &\leq C(s) \inf_{\lambda, \nu \in \mathbb{S}^1} \|\lambda q - \nu p\|_{W^{s,2}(B_0)} \inf_{\mu \in \mathbb{S}^1} \big( \|\lambda q - \mu\|_{W^{s,2}(B_0)} + \|\nu p - \mu\|_{W^{s,2}(B_0)} + 2\big) \\ \nonumber
        &= C(s) \inf_{\lambda \in \mathbb{S}^1} \|\lambda q - p\|_{W^{s,2}(B_0)} \inf_{\mu \in \mathbb{S}^1} \inf_{\nu \in \mathbb{S}^1} \big( \|\lambda q - \mu\|_{W^{s,2}(B_0)} + \|\nu p - \mu\|_{W^{s,2}(B_0)} + 2\big) \\ \nonumber
        &\leq C(s) \, d^s_{\ast} \big\vert_{B_0}(q, p) \big( 2 +  d^s_{\ast}\big\vert_{B_0}(1, q) + d^s_{\ast}\big\vert_{B_0}(1, p) \big) \\ \nonumber
        &\leq C(s, r) \, d^s_{\ast} \big\vert_{B_0} (q, p),
    \end{align}
    where in the last line we used (\ref{eqn:203}).
    Now we set $B_k = B_0 + k$ and see with Lemma \ref{lem:21} $(iv)$ and (\ref{eqn:204}) that
    \begin{align*}
        \|\rho - \eta\|_{H^s}^2 &\leq C(s) \sum_{k \in \Z} \||q|^2 - |p|^2\|_{W^{s,2}(B_k)}^2 \\
        &\leq C(s, r) \sum_{k \in \Z} d_{\ast}^s\big\vert_{B_k}(q, p)^2 \\
        &\leq C(s, r) \, d^s(q, p)^2.
    \end{align*}
    It remains to estimate $\|\phi' - \psi'\|_{H^{s-1}}$. Applying Lemma \ref{lem:2} yields
    \begin{align*}
        \|(\phi - \psi)'\|_{H^{s-1}}^2 &= \|(e^{i (\phi - \psi)})' e^{-i (\phi - \psi)}\|_{H^{s-1}}^2 \\
        &\leq C(s) \|(e^{i (\phi - \psi)})'\|_{H^{s-1}}^2 \big( \|e^{- i (\phi - \psi)}\|_{L^\infty}^2 + \|(e^{-i (\phi - \psi)})'\|_{H^{s-1}}^2 \big) \\
        &\leq C(s) \|(e^{i (\phi - \psi)})'\|_{H^{s-1}}^2 \big(1 + \|(e^{i (\phi - \psi)})'\|_{H^{s-1}}^2 \big) \,.
    \end{align*}
    It therefore suffices to derive the estimate for the quantity $\|(e^{i (\phi - \psi)})'\|_{H^{s-1}}^2$. Observe with Lemma \ref{lem:21} $(iv)$ that
    \begin{align*}
        \|(e^{i (\phi - \psi)})'\|_{H^{s-1}}^2 &\leq C(s) \sum_{k \in \Z} \inf_{\lambda \in \mathbb{S}^1} \|(e^{i (\phi - \psi)} - \lambda)'\|_{W^{s-1,2}(B_k)}^2 \\
        &= C(s) \sum_{k \in \Z} \inf_{\theta \in \R} \|e^{i (\phi - \psi - \theta)} - 1\|_{W^{s,2}(B_k)}^2 \\
        &= C(s) \sum_{k \in \Z} \inf_{\theta \in \R} \|e^{-i \psi} (e^{i (\phi + \theta)} - e^{i \psi})\|_{W^{s,2}(B_k)}^2 \,.
    \end{align*}
    We now carefully introduce the amplitudes:
    \begin{align*}
        &\quad\, \inf_{\theta \in \R} \|e^{-i \psi} (e^{i (\phi + \theta)} - e^{i \psi})\|_{W^{s,2}(B_k)}^2 \\
        &= \inf_{\theta \in \R} \left\| B e^{- i \psi(x)} B^{-1} \left( \frac{ A e^{i (\phi(x) + \theta)} - B e^{i \psi(x)}}{A} + B e^{i \psi(x)} \left(\frac{1}{A} - \frac{1}{B}\right) \right) \right\|_{W^{2,s}(B_k)}^2 \\
        &\leq C(s)\|B e^{-i \psi}\|_{W^{s,2}(B_k)}^2 \|B^{-1}\|_{W^{s,2}(B_k)}^2 \\
        &\times \left( \inf_{\theta \in \R} \left\|\frac{A e^{i (\phi(x) + \theta)} - B e^{i \psi(x)}}{A}\right\|_{W^{s,2}(B_k)}^2 + \left\|B e^{i \psi(x)} \left(\frac{1}{A} - \frac{1}{B}\right) \right\|_{W^{s,2}(B_k)}^2\right) \,.
    \end{align*}
    As $A, B > \delta > 0$, we can apply Lemma \ref{lem:212}. Together with (\ref{eqn:15}) we obtain
    \begin{equation*}
        \|A^{-1}\|_{W^{s,2}(B_k)}^2 \leq C(s, \delta) \|A\|_{W^{s,2}(B_k)}^2 \leq C(s, \delta, r) \big( \inf_{\lambda \in \mathbb{S}^1} \|A - \lambda\|_{W^{s,2}(B_k)}^2 + |B_k| \big) \leq C(s, \delta, r)\,,
    \end{equation*}
    and similarly $\|B^{\pm 1}\|_{W^{s,2}(B_k)}, \|q^{\pm 1}\|_{W^{s,2}(B_k)}, \|p^{\pm 1}\|_{W^{s,2}(B_k)} \leq C(s, \delta, r)$.
    We conclude again by reducing the situation to an application of Lemma \ref{lem:36} and the previously shown estimate (\ref{eqn:15}):
    \begin{align*}
        \allowdisplaybreaks
        \|(e^{i (\phi - \psi)})'\|_{H^{s-1}}^2 &\leq C(s) \sum_{k \in \Z} \|p\|_{W^{s,2}(B_k)}^2 \|B^{-1}\|_{W^{s,2}(B_k)}^2 \\
        &\times \Big(\inf_{\lambda \in \mathbb{S}^1} \|\lambda q - p\|_{W^{s,2}(B_k)}^2 \|A^{-1}\|_{W^{s,2}(B_k)}^2 \\
        & \qquad + \|A - B\|_{W^{s,2}(B_k)}^2 \|p\|_{W^{s,2}(B_k)}^2 \|A^{-1}\|_{W^{s,2}(B_k)}^2 \|B^{-1}\|_{W^{s,2}(B_k)}^2 \Big) \\
        &\leq C(s, \delta, r) \sum_{k \in \Z} d^s_\ast\big\vert_{B_k}(q, p)^2 \\
        &\leq C(s, \delta, r)\, d^s(q, p)^2\,.
    \end{align*}
\end{proof}

\subsection{Proof of Theorem \ref{thm:3} (ii)}

We assume 
\begin{equation*}
    \theta^s((1,0), (\rho, \del_x \phi)), \theta^s((1,0), (\eta, \del_x \psi)) < r\,,
\end{equation*}
and $\sqrt{\rho}, \sqrt{\eta} > \delta > 0$. We have to prove that
\begin{equation*}
    d^s(q, p) \leq C(s, r) \, \theta^s((\rho, \del_x \phi), (\eta, \del_x \psi)) \,.
\end{equation*}
As mentioned above, due to Lemma \ref{lem:403} it suffices to prove this with $\rho, \eta$ replaced by $A = \sqrt{\rho}$ and $B = \sqrt{\eta}$.
Recall the definitions (\ref{eqn:12}). We define
\begin{equation*}
    \tilde{d}^s(q, p) = \left(\int_{\R} \inf_{\lambda \in \mathbb{S}^1} \|\sqrt{\sech(y - \cdot)} (\lambda q - p)\|_{H^s}^2 \dd y \right)^{\frac12},
\end{equation*}
where we have replaced the $\sech$ in the definition of $d^s$ with $\sqrt{\sech}$. Some of the hard work for this direction has already been done in the proof of the following Lemma \ref{lem:3}. This was proven for $d^s$ in \cite[Lemma 6.1]{KochLiao}, but the proof is identical for $\tilde{d}^s$ as $\sqrt{\sech}$ is positive and still has sufficiently fast decay.

\begin{lemma}[{\cite[Lemma 6.1]{KochLiao}}] \label{lem:3}
    For all $s \geq 0$ the energy $E^s: X^s \longrightarrow \R_\geq$ is continuous with respect to $d^s$, and there exists $C(s) > 0$ so that
    \begin{equation*}
        d^s(1, q) \leq C(s) \sqrt{E^s(q)} \quad \text{ and } \quad \tilde{d}^s(1, q) \leq C(s) \sqrt{E^s(q)}
    \end{equation*}
    for all $q \in X^s$.
\end{lemma}
\begin{remark}
    The appearance of the square root is explained by a clash of notation: the energies $E^s$ as defined in \cite{KochLiao} correspond to $\sqrt{2 E^s}$ in our notation. 
\end{remark}

We first prove two Lemmas.

\begin{lemma} \label{lem:4}
    Let $s > \frac12$. There exists a constant $C(s) > 0$ so that for all $\phi \in \mathcal{S}' \cap H^s_{\loc}$ we have
    \begin{equation*}
        \|(e^{i \phi})'\|_{H^{s-1}} \leq C(s) (1 + \|\phi'\|_{H^{s-1}})^{\gamma} \|\phi'\|_{H^{s-1}} \,,
    \end{equation*}
    where $\gamma = 2s - 2$ if $s \geq 1$ and $\gamma = \frac{1 - s}{s - \frac12}$ if $s < 1$.
\end{lemma}
\begin{proof}
    We assume $\|\phi'\|_{H^{s-1}} \neq 0$. By Lemma \ref{lem:2} there exists a constant $C(s)$ so that
    \begin{equation} \label{eqn:23}
        \|\phi' e^{i \phi}\|_{H^{s-1}} \leq C(s) \|\phi'\|_{H^{s-1}} \big( \|e^{i \phi}\|_{L^\infty} + \|(e^{i \phi})'\|_{H^{s-1}} \big) \,.
    \end{equation}
    For $\epsilon \in (0, 1)$ and $f \in H^s_{\loc}$ define $f_\epsilon(x) = f(\epsilon x)$. This has the scaling estimates
    \begin{equation} \label{eqn:24}
        \min\{\epsilon^{s - \frac12}, \epsilon^{\frac12}\} \|f'\|_{H^{s-1}} \leq \|(f_\epsilon)'\|_{H^{s-1}} \leq \max\{\epsilon^{s - \frac12}, \epsilon^{\frac12}\} \|f'\|_{H^{s-1}} \,.
    \end{equation}
    Define $s_{\min} \leq s_{\max}$ so that $\{s_{\min}, s_{\max}\} = \{s - \frac12, \frac12\}$. Then we can rewrite the above as
    \begin{equation} \label{eqn:25}
        \epsilon^{s_{\max}} \|f'\|_{H^{s-1}} \leq \|(f_\epsilon)'\|_{H^{s-1}} \leq \epsilon^{s_{\min}} \|f'\|_{H^{s-1}} \,.
    \end{equation}
    We choose $\epsilon = (1 + 2 C(s) \|\phi'\|_{H^{s-1}})^{- \frac{1}{s_{\min}}}$ so that
    \begin{equation*}
        \|(\phi_\epsilon)'\|_{H^{s-1}} \leq \epsilon^{s_{\min}} \|\phi'\|_{H^{s-1}} = \frac{\|\phi'\|_{H^{s-1}}}{1 + 2 C(s) \|\phi'\|_{H^{s-1}}} \leq \frac{1}{2 C(s)} \,.
    \end{equation*}
    Combining this with (\ref{eqn:23}) yields
    \begin{equation*}
        \|(e^{i \phi_\epsilon})'\|_{H^{s-1}} \leq C(s) \|(\phi_\epsilon)'\|_{H^{s-1}} + \frac12 \|(e^{i \phi_\epsilon})'\|_{H^{s-1}} \,,
    \end{equation*}
    so we obtain
    \begin{equation*}
        \|(e^{i \phi_\epsilon})'\|_{H^{s-1}} \leq 2 C(s) \|(\phi_\epsilon)'\|_{H^{s-1}} \,.
    \end{equation*}
    We conclude with the scaling estimates (\ref{eqn:24}) that 
    \begin{equation*}
        \|(e^{i \phi})'\|_{H^{s-1}} \leq \epsilon^{- s_{\max}} \|(e^{i \phi_\epsilon})'\|_{H^{s-1}} 
        \leq 2 C(s) \epsilon^{- s_{\max}} \|(\phi_\epsilon)'\|_{H^{s-1}}
        \leq 2 C(s) \epsilon^{s_{\min} - s_{\max}} \|\phi'\|_{H^{s-1}} \,.
    \end{equation*}
    Lastly, note that
    \begin{equation*}
        \epsilon^{s_{\min} - s_{\max}} = \big(1 + 2 C(s) \|\phi'\|_{H^{s-1}}\big)^{\frac{|s - 1|}{s_{\min}}} = \big(1 + 2 C(s) \|\phi'\|_{H^{s-1}}\big)^{\gamma} \,.
    \end{equation*}
\end{proof}

\begin{lemma} \label{lem:5}
    Let $s > \frac12$ and $r > 0$. There exists $C(s, r) > 0$ so that for all $q = A e^{i \phi} \in \mathcal{S}' \cap H^s_{\loc}$ with $\theta^s((1,0), (A,\phi')) < r$ we have
    \begin{equation*}
        E^s(q) \leq C(s, r) \, \theta^s((1, 0), (A, \phi'))^2 \,.
    \end{equation*}
\end{lemma}
\begin{proof}
    For the amplitudinal part of the energy, we know from Lemma \ref{lem:403} that 
    \begin{equation*}
        \||q|^2 - 1\|_{H^{s-1}} \leq \|A^2 - 1\|_{H^s} \leq C(s, r) \|A - 1\|_{H^s} \,.
    \end{equation*}
    For the remainder, we use Lemma \ref{lem:2}:
    \begin{align*}
        \|q'\|_{H^{s-1}} &\leq \|A' e^{i \phi}\|_{H^{s-1}} + \|A (e^{i \phi})'\|_{H^{s-1}} \\
        &\leq C(s) \|A'\|_{H^{s-1}} \big( \|e^{i\phi}\|_{L^\infty} + \|(e^{i \phi})'\|_{H^{s-1}} \big) \\
        &+ \|(e^{i \phi})'\|_{H^{s-1}} \big( \|A - 1\|_{L^\infty} + 1 + \|A'\|_{H^{s-1}}  \big) \,.
    \end{align*}
    We now conclude by estimating both appearances of $\|(e^{i \phi})'\|_{H^{s-1}}$ with Lemma \ref{lem:4}.
\end{proof}

\begin{proof}[Proof of Theorem \ref{thm:3} (ii)]
    We split the distance $d^s(q, p)$ into two parts:
    \begin{align*}
        d^s(q, p)^2 &\leq 2 \int_{\R} \inf_{\theta \in \R} \|\sech(y - \cdot) A (e^{i (\phi + \theta)} - e^{i \psi})\|_{H^s}^2 \dd y \\
        &+ 2 \int_{\R} \|\sech(y - \cdot) (B - A) e^{i \psi}\|_{H^s}^2 \dd y \\
        &= (I) + (II) \,.
    \end{align*}
    We use the algebra property of $H^s$ and Lemma \ref{lem:3} to estimate
    \begin{align*}
        (I) &\leq C(s) \sup_{y \in \R} \Big\|\sqrt{\sech(y - \cdot)} A e^{i \psi} \Big\|_{H^s}^2 \tilde{d}^s(1, e^{i (\phi - \psi)})^2 \\
        &\leq C(s, r) \big(\|A - 1\|_{H^s}^2 + 1 \big) \sup_{y \in \R} \Big\|\sqrt{\sech(y - \cdot)} e^{i \psi} \Big\|_{H^s} E^s(e^{i (\phi - \psi)}) \,.
    \end{align*}
    With Lemma \ref{lem:5} we can estimate $E^s(e^{i (\phi - \psi)})$ by $\|\phi' - \psi'\|_{H^{s-1}}^2$, and Lemma \ref{lem:2} yields
    \begin{equation*}
        \sup_{y \in \R} \Big\|\sqrt{\sech(y - \cdot)} e^{i \psi} \Big\|_{H^s} \leq C(s) \sup_{y \in \R} \Big\|\sqrt{\sech(y - \cdot)} \Big\|_{H^s} (\|e^{i \psi}\|_{L^\infty} + \|(e^{i \psi})'\|_{H^{s-1}}) \leq C(s, r) \,.
    \end{equation*}
    It follows that
    \begin{equation*}
        (I) \leq C(s, r) \, \theta^s((A, \phi'), (B, \psi'))^2 \,.
    \end{equation*}
    Note that
    \begin{align*}
        (II) &\leq C(s) \int_{\R} \Big\|\sqrt{\sech(y - \cdot)} (A - B) \Big\|_{H^s}^2 \Big( 1 + \inf_{\lambda \in \mathbb{S}^1} \Big\|\sqrt{\sech(y - \cdot)} (e^{i \psi} - \lambda) \Big\|_{H^s}^2 \Big) \dd y \\
        &\leq C(s) \left( \int_{\R} \Big\|\sqrt{\sech(y - \cdot)} (A - B) \Big\|_{H^s}^2 \dd y 
        + \big\|\sqrt{\sech}\big\|_{H^s}^2 \|A - B\|_{H^s}^2 \, \tilde{d}^s(1, e^{i \psi})^2 \right) \,.
    \end{align*}
    We can deal with the second term as before. For the first one, we use Lemma \ref{lem:21} $(iv)$ and Young's convolution inequality:
    \begin{align*}
        \int_{\R} \Big\|\sqrt{\sech(y - \cdot)} (A - B) \Big\|_{H^s}^2 \dd y 
        &\leq \sum_{k \in \Z} \sup_{y \in [k, k+1]} \sum_{j \in \Z} \Big\|\sqrt{\sech(y - \cdot)} (A - B) \Big\|_{W^{s,2}([j, j+3])}^2 \\
        &\leq \sum_{j, k \in \Z} \big\|\sqrt{\sech} \big\|_{C^{\lceil s \rceil+1}([k-j-3, k-j+1]))}^2 \|A - B\|_{W^{s,2}([j,j+3])}^2 \\
        &\leq \sum_{k \in \Z} \big\|\sqrt{\sech} \big\|_{C^{\lceil s \rceil+1}([k-3, k+1]))}^2 \sum_{j \in \Z} \|A - B\|_{W^{s,2}([j, j+3])}^2 \\
        &\leq C(s) \|A - B\|_{H^s}^2 \,.
    \end{align*}
    Therefore
    \begin{equation*}
        (II) \leq C(s, r) \, \theta^s((A, \phi), (B, \psi))^2 \,.
    \end{equation*}
    To conclude, we have shown that
    \begin{equation*}
        d^s(q, p)^2 \leq (I) + (II) \leq C(s, r) \, \theta^s((A, \phi), (B, \psi))^2 \,.
    \end{equation*}
\end{proof}

\begin{remark} \label{rmk:2}
    Recall the definition of $d^s_\ast\big\vert_B$ (see (\ref{eqn:19})). We have shown in particular that there exist constants such that
    \begin{equation*}
        \left(\sum_{k \in \Z} d^s_\ast\big\vert_{B_k}(q, p)^2 \right)^{\frac12} \leq C(s, r)\, d^s(q, p) \leq C(s, \delta, r) \left(\sum_{k \in \Z} d^s_\ast\big\vert_{B_k}(q, p)^2 \right)^{\frac12}
    \end{equation*}
    for all $q, p \in X^s$ with $|q|, |p| > \delta > 0$ and $d^s(1, q), d^s(1, p) < r$. Here the first estimate is Lemma \ref{lem:36}, while the second estimate actually follows from $(ii)$ together with the fact that we showed $(i)$ by proving (\ref{eqn:21}).
\end{remark}

Let us say a few words on how Corollary \ref{cor:1} follows from Theorem \ref{thm:3}.
\begin{proof}[Proof of Corollary \ref{cor:1}]
    The Madelung transform is well-defined on equivalence classes under multiplication by $\mathbb{S}^1$, as $v = \phi'$ ignores changes by a constant in the phase $\phi$. Note also that $s \geq 1$, and so for any $(\rho, v) \in \mathcal{Y}^s$ we have $v \in L^2 \subset L^1_{\loc}$. Therefore we can define
    \begin{equation*}
        \phi(x) = \int_0^x v(y) \dd y \,.
    \end{equation*}
    Recall that $b < \frac43$ and $\epsilon < \epsilon_0(\mu)$ (see (\ref{eqn:156})). Due to (\ref{eqn:157}) and (\ref{eqn:158}), there exists $\delta > 0$ such that $|q| > \delta$ for all $q \in X^s$ with $E(q) < b$ or $E^\mu(q) < \epsilon$. With Lemma \ref{lem:3} we find some $r = r(s, \epsilon, b) > 0$ such that $d^s(1, q) < r$.
    Then Theorem \ref{thm:3} establishes the bilipschitz estimates.
\end{proof}

\section{Proof of Theorem \ref{thm:2}}

Given that Theorem \ref{thm:3} establishes an equivalence between the relevant function spaces $(X^s, d^s)$ and $(\mathcal{Y}^s, \theta^s)$, the proof of Theorem \ref{thm:2} is now primarily a matter of carefully carrying over the results of Theorem \ref{thm:1}. This is straightforward for the existence and continuity results. Uniqueness requires a further Lemma.

\begin{lemma} \label{lem:71}
    Let $I \ni 0$ be an open time interval and $q_0 \in L^\infty \cap \dot{H}^1$. Suppose 
    \begin{equation*}
        q_1, q_2 \in C(I; L^2_{\loc}) \cap L^\infty(I; L^\infty \cap \dot{H}^1)
    \end{equation*} are two distributional solutions to (GP) with $q_1(0) = q_2(0) = q_0$. Then $q_1 = q_2$.
\end{lemma}
\begin{proof}
    See Appendix C.
\end{proof}

This result is necessary because Theorem \ref{thm:1}, in the way it is stated in \cite{KochLiao}, only yields uniqueness for the following class of solutions, which for the case $s \geq 1$ is a priori smaller.
\begin{definition}[Solutions to (GP) \cite{KochLiao}] \label{def:1}
    Let $s \geq 0$. We say that $q \in C(I; X^s)$ is a solution of the Gross-Pitaevskii equation (GP) with initial data $q_0 \in X^s$ on the open time interval $I \ni 0$ if there exists $\tilde{q}: I \longrightarrow H^s_{\loc}$ such that the following hold:
    \begin{enumerate}[(i)]
        \item $\tilde{q}$ solves (GP) in the sense of distributions on $I \times \R$.
        \item $\tilde{q}$ projects to $q$, which means that $\tilde{q} \mathbb{S}^1 = q$.
        \item We have
        \begin{equation*}
            \big[t \mapsto \tilde{q}(t) - \tilde{q}(0)\big] \in C(I; L^2(\R)) \,.
        \end{equation*}
        \item For all compact intervals $[a, b] \subset I$ and for some (and hence for all) regularized initial data $\tilde{q}_0^\ast$ of $\tilde{q}(0)$ we have
        \begin{equation*}
            \big[t \mapsto \tilde{q}(t) - \tilde{q}^\ast_0 \big] \in L^4([a, b] \times \R) \,.
        \end{equation*} 
    \end{enumerate}
\end{definition}

The uniqueness result in Theorem \ref{thm:1} for $s \geq 1$ is therefore weaker than the one in Lemma \ref{lem:71}. The proofs, however, are almost identical: in \cite{KochLiao} uniqueness is shown by a classical argument with an energy estimate and Gr\"onwall's inequality. We extend this argument for $s \geq 1$ to gain Lemma \ref{lem:71}.

\begin{remark}
    If $\tilde{p} \in C(I; L^2_{\loc}) \cap L^\infty(I; L^\infty \cap \dot{H}^1)$ is a distributional solution to (GP), as in Lemma \ref{lem:71}, with initial data $\tilde{p}(0) \mathbb{S}^1 \in X^1$, then $\tilde{p} \mathbb{S}^1$ is also a solution in the sense of Definition \ref{def:1}. The reason is that by Theorem \ref{thm:1} there exists a solution $q \in C(I; X^1)$ in the sense of Definition \ref{def:1} with initial data $q(0) = \tilde{p}(0) \mathbb{S}^1$. One can see that this has a representative $\tilde{q} \in C(I; L^2_{\loc}) \cap L^\infty(I; L^\infty \cap \dot{H}^1)$ which solves (GP) in distribution, so Lemma \ref{lem:71} implies $\tilde{q} = \tilde{p}$.
\end{remark}

Theorem \ref{thm:2} states that (hGP) is globally-in-time well-posed, meaning that there exist solutions, they are unique, and the flow map is continuous. The structure of the proof is to transfer the existence and continuity result for (GP) from Theorem \ref{thm:1} via the Madelung transform over to (hGP). This requires the absence of vacuum, which we obtain by the energy assumptions $E < \frac43$ or $E^\mu < \epsilon_0(\mu)$ (see (\ref{eqn:157}) and (\ref{eqn:158})). Uniqueness for (hGP) is similarly inferred from the uniqueness result for (GP) in Lemma \ref{lem:71}.

Recall that by Lemma \ref{lem:3} the energy functionals
$E^s: X^s \longrightarrow \R_\geq$ are continuous. Recall furthermore the definitions (\ref{eqn:12}).

\begin{proof}[Proof of Theorem \ref{thm:2}]

    \textbf{Existence.}
    We are given an initial data $(\rho_0, v_0) \in \mathcal{Y}^s$ which fulfills one of the bounds $\mathcal{E}(\rho_0, v_0) < \frac43$ or $\mathcal{E}^\mu(\rho_0, v_0) < \epsilon_0(\mu)$. We define $q_0 = \mathcal{M}^{-1}(\rho_0, v_0)$ and obtain via Theorem \ref{thm:1} a solution $q \in C_b(\R; X^s)$ of (GP) in the sense of Definition \ref{def:1}. 
    Our solution $q$ has a special representative $\tilde{q} \in \mathcal{S}'(\R \times \R)$.
    In both cases $E(q_0) < \frac43$ and $E^\mu(q_0) < \epsilon_0(\mu)$, we obtain either (\ref{eqn:158}) or (\ref{eqn:157}), so there exists some $\delta > 0$ depending on the initial data such that $|\tilde{q}| > \delta > 0$. Now Corollary \ref{cor:1} implies $(\rho, v) = \mathcal{M}(\tilde{q}) \in C_b(\R; \mathcal{Y}^s)$.
    \begin{center}
        \begin{tikzpicture}
            \node (A) at (0,0) {$(\rho_0, v_0)$};
            \node (B) at (3,0) {$(\rho, v)(t)$};
            \node (C) at (0,-2) {$q_0$};
            \node (D) at (3,-2) {$q(t)$};
            \draw[->, dash pattern=on 2.0pt off 0.8pt, decorate,
            decoration={
                snake,
                amplitude = .4mm,
                segment length = 2mm,
                post length=0.9mm}] (A) -- node[above]{(hGP)} (B);
            \draw[|->] (A) -- node[left]{$\mathcal{M}^{-1}$} (C);
            \draw[->, decorate, decoration={
                snake,
                amplitude = .4mm,
                segment length = 2mm,
                post length=0.9mm}]
            (C) -- node[below]{(GP)} (D);
            \draw[|->] (D) -- node[right]{$\mathcal{M}$} (B);
        \end{tikzpicture}
    \end{center}
    
    We show that $(\rho, v)$ is a distributional solution of (hGP) in the sense of Definition \ref{def:2}.
    We fix a ball $B_0 \subset \R$ and a time interval $J = (a, b) \subset \R$ with $0 \in J$. It suffices to verify that (hGP) holds in distribution, i.e. when tested against any test function $f \in \mathcal{D}(J \times B_0)$. 
    
    \textbf{On regularity.}
    Due to Lemmas \ref{lem:212} and \ref{lem:36} for $s \geq 1$, we know that $\tilde{q}, \tilde{q}^{-1} \in L^\infty(J; W^{1,2}(B_0))$.
    From these considerations $\del_{xx} \tilde{q} \in L^\infty(J; W^{-1,2}(B_0))$ and $(|\tilde{q}|^2 - 1) \tilde{q} \in L^\infty(J; W^{1,2}(B_0))$ directly follow. Then $\del_t \tilde{q} \in L^\infty(J; W^{-1,2}(B_0))$ holds because $\tilde{q}$ solves $(GP)$ in the sense of distributions. 
    
    As a consequence of duality and the algebra property of $H^1$, one obtains the product estimate
    $\|f g\|_{H^{-1}} \leq C \|f\|_{H^1} \|g\|_{H^{-1}} $.
    From this we obtain some regularity for some of the more difficult terms appearing in the subsequent calculations, for example $\del_t \tilde{q} \overline{\tilde{q}},\, \del_{xx} \tilde{q} \overline{\tilde{q}} \in L^\infty(J; W^{-1,2}(B_0))$. 
    We now present approximation arguments that derive (hGP)\textsubscript{1} and (hGP)\textsubscript{2} from (GP).
    
    \textbf{Obtaining (hGP)\textsubscript{1} from (GP)}.
    Set $\tilde{q}_\epsilon = \eta_\epsilon \ast \tilde{q}$ for a standard mollifier $(\eta_\epsilon)_{\epsilon > 0}$, i.e. some $\eta_\epsilon(x) = \eta(\epsilon^{-1}(\epsilon^{-1} x))$ where $\eta \in C_c^\infty(\R;\R_\geq)$ with $\int \eta \dd x = 1$. Note that $|\tilde{q}| > \delta$ implies $|\tilde{q}_\epsilon| > \frac{\delta}{2}$ for sufficiently small $\epsilon > 0$ as we have sufficient regularity. We define $\rho_\epsilon = |\tilde{q}_\epsilon|^2$ and $v_\epsilon = \Imag\big[\frac{\del_x \tilde{q}_\epsilon}{\tilde{q}_\epsilon}\big]$.
    Note furthermore the identity $\frac{\del_x \tilde{q}_\epsilon}{\tilde{q}_\epsilon} = \frac12 \frac{\del_x \rho_\epsilon}{\rho_\epsilon} + i v_\epsilon$, which we use below.
    Equation (hGP)\textsubscript{1} can be obtained by multiplying (GP) for $\tilde{q}_\epsilon$ with $\overline{\tilde{q}}_\epsilon$, taking the imaginary part, and then the limit:
    \begin{align*}
        \allowdisplaybreaks
        0 = \Imag \left(\overline{\tilde{q}} \, \text{(GP)} \right) \xleftarrow{\epsilon \rightarrow 0}\, &\Imag \left[ i \del_t \tilde{q}_\epsilon \overline{\tilde{q}_\epsilon} + \del_{xx} \tilde{q}_\epsilon \overline{\tilde{q}_\epsilon} - 2 \tilde{q}_\epsilon (|\tilde{q}_\epsilon|^2 - 1) \overline{\tilde{q}_\epsilon} \right] \\
        =\, &\Real \left[ \del_t \tilde{q}_\epsilon \overline{\tilde{q}_\epsilon} \right] + \del_x \Imag\left[\frac{\del_x \tilde{q}_\epsilon}{\tilde{q}_\epsilon} \tilde{q}_\epsilon \overline{\tilde{q}_\epsilon} \right] - \Imag \left[ \del_x \tilde{q}_\epsilon \del_x \overline{\tilde{q}_\epsilon} \right] - \Imag \left[2 |\tilde{q}_\epsilon|^2 (|\tilde{q}_\epsilon|^2 - 1) \right] \\
        =\, &\frac12 \del_t (|\tilde{q}_\epsilon|^2) + \del_x \left( |\tilde{q}_\epsilon|^2 \Imag\left[ \frac{\del_x \tilde{q}_\epsilon}{\tilde{q}_\epsilon} \right] \right) \\
        =\, &\frac12 \del_t \rho_\epsilon + \del_x (\rho_\epsilon v_\epsilon) \\
        & \hspace{-1.1em} \xrightarrow{\epsilon \rightarrow 0} \frac12 \del_t \rho + \del_x(\rho v) \,.
    \end{align*}
    We have to justify the limits in distribution on both sides.
    Observe that
    \begin{align*}
        \left|\int_J \int_{B_0} (\del_t \tilde{q}_\epsilon \overline{\tilde{q}_\epsilon} - \del_t \tilde{q} \overline{\tilde{q}}) \overline{f} \right|
        &\lesssim \|\del_t \tilde{q}_\epsilon - \del_t \tilde{q}\|_{L^\infty(J; W^{-1,2}(B_0))} \|\tilde{q}_\epsilon\|_{L^\infty(J; W^{1,2}(B_0))} \|f\|_{L^\infty(J; W^{1,2}(B_0))} \\
        &+\|\del_t \tilde{q}\|_{L^\infty(J; W^{-1,2}(B_0))} \|\tilde{q}_\epsilon - \tilde{q}\|_{L^\infty(J; W^{1,2}(B_0))} \|f\|_{L^\infty(J; W^{1,2}(B_0))} \\
        &\xrightarrow{\epsilon \rightarrow 0} 0 \,.
    \end{align*}
    With the same estimates, we can take the limit of the distribution $\del_{xx} \tilde{q}_\epsilon \overline{\tilde{q}_\epsilon}$. The convergence of the nonlinear term follows similarly.
    We have shown that
    \begin{equation*}
        i \del_t \tilde{q}_\epsilon \overline{\tilde{q}_\epsilon} + \del_{xx} \tilde{q}_\epsilon \overline{\tilde{q}_\epsilon} - 2 \tilde{q}_\epsilon (|\tilde{q}_\epsilon|^2 - 1) \overline{\tilde{q}_\epsilon} \xrightarrow{\epsilon \rightarrow 0} \overline{\tilde{q}} \big( i \del_t \tilde{q}  + \del_{xx} \tilde{q} - 2 \tilde{q} (|\tilde{q}|^2 - 1) \big) = 0
    \end{equation*}
    in distribution on $J \times B_0$.
    As
    \begin{align*}
        & \quad\, \left|\int_J \int_{B_0} (\rho_\epsilon v_\epsilon - \rho v) \overline{\del_x f} \dd x \dd t \right| \\ &\lesssim  \big(\|\tilde{q}\|_{L^\infty(J \times B_0)} + \|\tilde{q}_\epsilon\|_{L^\infty(J \times B_0)}\big) \|\tilde{q}_\epsilon - \tilde{q}\|_{L^\infty(J \times B_0)} \|v\|_{L^\infty(J; L^2(B_0))}\|\del_x f\|_{L^2(J \times B_0)} \\
        &+ \big(\|\tilde{q}\|_{L^\infty(J \times B_0)} + \|\tilde{q}_\epsilon\|_{L^\infty(J \times B_0)}\big) \|\tilde{q}_\epsilon\|_{L^\infty(J \times B_0)} \|v_\epsilon - v\|_{L^\infty(J; L^2(B_0))} \|\del_x f\|_{L^2(J \times B_0)} \xrightarrow{\epsilon \rightarrow 0} 0 \,,
    \end{align*}
    and we can similarly show $\del_t \rho_\epsilon \xrightarrow{\epsilon \rightarrow 0} \del_t \rho$, we have
    \begin{equation*}
        \frac12 \del_t \rho_\epsilon + \del_x (\rho_\epsilon v_\epsilon) \xrightarrow{\epsilon \rightarrow 0} \frac12 \del_t \rho + \del_x (\rho v)
    \end{equation*}
    in distribution on $J \times B_0$.
    
    \textbf{Obtaining (hGP)\textsubscript{2} from (GP)}. We now repeat these arguments for the second equation (hGP)\textsubscript{2}. Here we divide (GP) for $\tilde{q}_\epsilon$ by $\tilde{q}_\epsilon$, take a further derivative, the real part, and then the limit:
    \begin{align*}
        \allowdisplaybreaks
        0 = \Real\del_x \left(\frac{\text{(GP)}}{\tilde{q}}\right)\xleftarrow{\epsilon \rightarrow 0} &\, \Real \left[ \del_x \left( \frac{i \del_t \tilde{q}_\epsilon}{\tilde{q}_\epsilon} + \frac{ \del_{xx} \tilde{q}_\epsilon}{\tilde{q}_\epsilon} - 2 (|\tilde{q}_\epsilon|^2 - 1) \right) \right] \\
        = &\, - \Imag \left[ \del_x \del_t (\ln \tilde{q}_\epsilon)\right] + \Real \left[ \del_x \left( \del_x \left( \frac{\del_x \tilde{q}_\epsilon}{\tilde{q}_\epsilon} \right) + \left( \frac{\del_x \tilde{q}_\epsilon}{\tilde{q}_\epsilon} \right)^2 \right) \right] - 2 \del_x (|\tilde{q}_\epsilon|^2) \\
        = &\, - \del_t v_\epsilon - \del_x(v_\epsilon^2) - 2 \del_x \rho_\epsilon +  \del_x \left( \del_x \left( \frac12 \frac{\del_x \rho_\epsilon}{\rho_\epsilon} \right) + \left( \frac12 \frac{\del_x \rho_\epsilon}{\rho_\epsilon} \right)^2 \right) \\
        &\hspace{-1.1em} \xrightarrow{\epsilon \rightarrow 0}
        - \del_t v - \del_x(v^2) - 2 \del_x \rho +  \del_x \left( \del_x \left( \frac12 \frac{\del_x \rho}{\rho} \right) + \left( \frac12 \frac{\del_x \rho}{\rho} \right)^2 \right) \,.
    \end{align*}
    Of course, the limits have to be justified again. For the left-hand side, we can proceed just as before since $\tilde{q}^{-1} \in L^\infty(J; W^{1,2}(B_0))$. On the right hand side the difficult terms are $v^2$ and $\left(\frac12 \frac{\del_x \rho}{\rho}\right)^2$, as here the square of a distribution in $H^{s-1}$ is taken. The situation would be much more difficult if we did not assume $s \geq 1$. In our case we indeed have $v, \frac{\del_x \rho}{\rho} \in C(J; L^2(B_0))$, which implies that the squares are trivially defined. Furthermore
    \begin{align*}
        &\quad\, \left| \int_J \int_{B_0} (v_\epsilon^2 - v^2) \del_x f \dd x \right| \\
        &\leq \|v_\epsilon - v\|_{L^\infty(J; L^2(B_0))} \big(\|v_\epsilon\|_{L^\infty(J; L^2(B_0))}  + \|v\|_{L^\infty(J; L^2(B_0))}\big) \|\del_x f\|_{L^\infty(J \times B_0)} \xrightarrow{\epsilon \rightarrow 0} 0 \,,
    \end{align*}
    and the same estimate works for $\left( \frac12 \frac{\del_x \rho}{\rho} \right)^2$. The remaining terms are strictly easier to deal with.
    
    \textbf{Uniqueness.}
    Let $0 \in I \subset \R$ be a bounded open interval and let $(\rho_1, v_1), (\rho_2, v_2) \in C(I; \mathcal{Y}^s)$ be two solutions to (hGP) in the sense of the theorem, both with initial data $(\rho_0, v_0) \in \mathcal{Y}^s$. In particular, they satisfy one of the energy bounds $\mathcal{E} < b < \frac43$ or $\mathcal{E}^\mu < c(\mu) \epsilon < c(\mu) \epsilon_0(\mu)$ (see (\ref{eqn:7}) and (\ref{eqn:156})). As before, this implies that there exists a $\delta > 0$ so that $\sqrt{\rho_k} > \delta$, where $k \in \{0, 1, 2\}$. Since $v_k \in L^2 \subset L^1_{\loc}$, we can define $\phi_k(x) = \int_0^x v_k(y) \dd y$ and $\tilde{q}_k = \sqrt{\rho_k} e^{i \phi_k}$. Note that $\tilde{q}_k$ having uniformly bounded energy $E^1$ implies $\tilde{q}_k \in L^\infty(I; L^\infty \cap \dot{H}^1)$. We now fix $j \in \{1, 2\}$.
    Writing $q_j = \tilde{q}_j \mathbb{S}^1$ for the equivalence class, we know from Corollary \ref{cor:1} that $q_j \in C(I; X^s)$. 
    
    Just as in the existence part of the proof, one can show that $(\rho_j, v_j)$ solving (hGP) implies that for the quantity
    \begin{equation} \label{eqn:160}
        Q_j = i \del_t \tilde{q}_j + \del_{xx} \tilde{q}_j - 2 \tilde{q}_j (|\tilde{q}_j|^2 - 1)
    \end{equation}
    we have
    \begin{equation} \label{eqn:159}
        \Imag\left[Q_j \overline{\tilde{q}_j} \,\right] = 0 \qquad \text{ and } \qquad \del_x \Real\left[ \frac{Q_j}{\tilde{q}_j}\right] = 0
    \end{equation}
    in the sense of distributions.
    We sketch the argument that follows with a diagram.
    
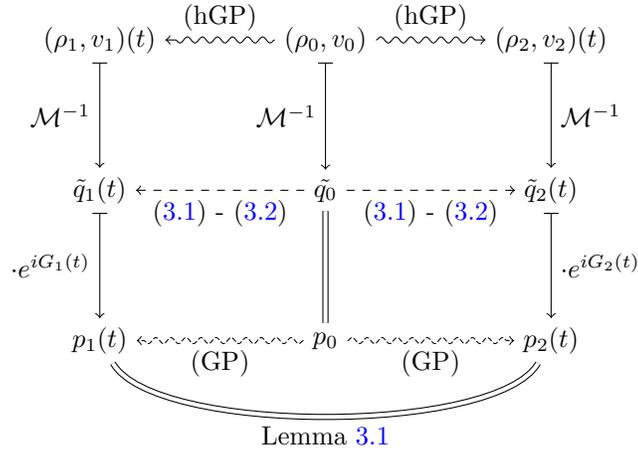
\begin{center}
        \begin{tikzpicture}
            \node (E) at (-3,0) {$(\rho_1, v_1)(t)$};
            \node (F) at (-3,-2) {$\tilde{q}_1(t)$};
            \node (A) at (0,0) {$(\rho_0, v_0)$};
            \node (B) at (3,0) {$(\rho_2, v_2)(t)$};
            \node (C) at (0,-2) {$\tilde{q_0}$};
            \node (D) at (3,-2) {$\tilde{q}_2(t)$};
            \node (G) at (0, -4) {$p_0$};
            \node (H) at (3, -4) {$p_2(t)$};
            \node (I) at (-3, -4) {$p_1(t)$};
            \draw[->, decorate, 
            decoration={
                snake,
                amplitude = .4mm,
                segment length = 2mm,
                post length=0.9mm}] 
            (A) -- node[above]{(hGP)} (B);
            \draw[->, decorate, 
            decoration={
                snake,
                amplitude = .4mm,
                segment length = 2mm,
                post length=0.9mm}] 
            (A) -- node[above]{(hGP)} (E);
            \draw[|->] (A) -- node[left]{$\mathcal{M}^{-1}$} (C);
            \draw[|->] (E) -- node[left]{$\mathcal{M}^{-1}$} (F);
            \draw[|->] (B) -- node[right]{$\mathcal{M}^{-1}$} (D);
            \draw[->, dashed] (C) -- node[below]{(\ref{eqn:160}) - (\ref{eqn:159})} (D);
            \draw[->, dashed] (C) -- node[below]{(\ref{eqn:160}) - (\ref{eqn:159})} (F);
            \draw[|->] (F) -- node[left]{$\cdot e^{i G_1(t)}$} (I);
            \draw[-, double equal sign distance] (C) -- node[left]{} (G);
            \draw[|->] (D) -- node[right]{$\cdot e^{i G_2(t)}$} (H);
            \draw[->, dash pattern=on 2.0pt off 0.8pt, decorate, 
            decoration={
                snake,
                amplitude = .4mm,
                segment length = 2mm,
                post length=0.9mm}] (G) -- node[below]{(GP)} (I);
            \draw[->, dash pattern=on 2.0pt off 0.8pt, decorate, 
            decoration={
                snake,
                amplitude = .4mm,
                segment length = 2mm,
                post length=0.9mm}] (G) -- node[below]{(GP)} (H);
            \draw[-, double equal sign distance] (I) to[out=-60,in=-120,looseness=0.5] node[below] {Lemma \ref{lem:71}} (H);
        \end{tikzpicture}
    \end{center}
    Due to (\ref{eqn:159}) we have in particular
    \begin{equation*}
        \Imag\left[\frac{Q_j}{\tilde{q}_j}\right] = \Imag\left[\frac{Q_j \overline{\tilde{q}_j}}{|\tilde{q}_j|^2}\right]
        = \frac{\Imag [Q_j \overline{\tilde{q}_j}]}{|\tilde{q}_j|^2} = 0 \,,
    \end{equation*}
    and hence for every $t \in I$ there exists a $g_j(t) \in \R$ so that
    \begin{equation*}
        g_j(t) = \frac{Q_j}{\tilde{q}_j} \,.
    \end{equation*}
    We see that, in fact, $\tilde{q}_j$ does not necessarily solve (GP). The reason is that for each time $t \in I$ we had to make an arbitrary choice of a constant-in-space phase rotation, as this information is lost in the Madelung transform. This choice was the arbitrary lower limit $0$ in the integral $\phi(t) = \int_0^t v(s) \dd s$.
    In order to find solutions to (GP), we would now like to define
    \begin{equation*}
        p_j(t) = e^{i G_j(t)} \tilde{q}_j(t) \quad \text{ where } \quad G_j(t) = \int_0^t g_j(s) \dd s \,.
    \end{equation*}
    Then
    \begin{equation*}
        i \del_t p_j + \del_{xx} p_j - 2 p_j (|p_j|^2 - 1) = Q_j - G'(t) \tilde{q}_j = 0 \,.
    \end{equation*}
    This argument requires $g_j: I \longrightarrow \R$ to be locally integrable. We show that $g_j \in C(I; \R)$ by verifying that $Q_j \in C(I; W^{-1,2}(B_0))$ for any ball $B_0 \subset \R$. With the same reasoning as in the existence part of the proof, $\rho_j \in C(I; W^{1,2}(B_0))$ and $v_j \in C(I; L^2(B_0))$ solving (hGP) in distribution implies $\rho_j \in C^1(I; W^{-1,2}(B_0))$ and $v_j \in C^1(I; W^{-1,1}(B_0))$. In particular we have $\del_t \phi_j \in C(I; L^1(B_0))$. Observe that
    \begin{align*}
        Q_j = i \frac{\del_t \rho_j}{\rho_j} \tilde{q}_j + i (\del_t \phi_j) \tilde{q}_j + \del_{xx} \tilde{q}_j - 2 \tilde{q}_j (|\tilde{q}_j|^2 - 1) \,.
    \end{align*}
    Verifying the products of distributions, each term can now be seen to be in $C(I; W^{-1, 2}(B_0))$. 
    We have shown that for any bounded interval $I \ni 0$, both $p_1$ and $p_2$ are distributional solutions to (GP) with initial data $p_1(0) = p_2(0) = \tilde{q}_0$. At the same time $p_j \in C(I; L^2_{\loc}) \cap L^\infty(I; L^\infty_{t,x} \cap \dot{H}^1)$. Therefore Lemma \ref{lem:71} implies $p_1 = p_2$, from which $q_1 = q_2$ in $C(I; X^s)$ and $(\rho_1, v_1) = (\rho_2, v_2)$ follow.
    
    \textbf{Continuity.}
    This is a direct consequence of the continuity result for (GP) from Theorem \ref{thm:1}, the continuity of the energy functionals from Lemma \ref{lem:3}, and the local bilipschitz equivalence from Theorem \ref{thm:3}.
\end{proof}
    
\appendix

\section{Absence of vacuum for small energies}

\begin{lemma} \label{lem:0}
    For $\delta \in [0, 1]$ and $s \in \big(\frac12, 1\big]$ define
    \begin{equation*}
        E^s_\delta = \inf\big\{E^s(q): q \in H^s_{\loc}, \inf_{x \in \R} |q(x)| \leq \delta\big\} \,.
    \end{equation*}
    Then $E^s_1 = 0$, the function $\delta \mapsto E^s_\delta$ is decreasing, and there exists a constant $\tilde{C}(s) > 0$ so that
    \begin{equation} \label{eqn:260}
        E^s_\delta \geq \frac{(1 - \delta)^2}{\tilde{C}(s)}\,.
    \end{equation}
    
    Assume $s = 1$ and write $E_\delta = E^1_\delta$.
    Set $q_0 = \tanh$, $q_1 = 1$, and for $\delta \in (0, 1)$ define
    \begin{equation} \label{eqn:27}
        q_\delta(x) = \tanh\big(|x| + \tanh^{-1}(\delta)\big) \,.
    \end{equation}
    We have 
    \begin{equation*}
        E_\delta = E(q_\delta) = \frac43 - 2 \delta + \frac23 \delta^3 \,.
    \end{equation*}
    There exists a strictly decreasing inverse function $\tilde{\delta}: [0,\frac43] \longrightarrow [0,1]$ with $\tilde{\delta}(0) = 1$, $\tilde{\delta}(\frac43) = 0$ and
    \begin{equation*}
        \tilde{\delta}(b) = \inf\big\{ \inf_{x \in \R} |q(x)|: q \in H^1_{\loc}, E(q) \leq b \big\}.
    \end{equation*}
\end{lemma}
\begin{proof}
    We see that $E^s_1 = 0$ by choosing $q = 1$. Clearly the set over which the infimum is taken increases with $\delta$, and hence the infimum is decreasing.
    Recall that Lemma \ref{lem:36} implies
    \begin{equation*}
        \inf_{\lambda \in \mathbb{S}^1} \|q - \lambda\|_{W^{s,2}(B_k)} \leq C(s) \,d^s(1, q)\,,
    \end{equation*}
    where $B_k = B_0 + k, k \in \Z$ are balls of radius $1$. Estimating with Lemma \ref{lem:3} on the right and the Sobolev embedding $W^{1,2}(B_k) \xhookrightarrow{\quad} L^\infty$ on the left, we obtain
    \begin{equation*}
        1 - \delta \leq \sup_{k \in \Z} \inf_{\lambda \in \mathbb{S}^1} \|q - \lambda\|_{L^\infty(B_k)} \leq C(s) \sqrt{E^s(q)}
    \end{equation*}
    for every $q \in H^s_{\loc}$ with $\inf_{x \in \R} |q(x)| \leq \delta$. This proves (\ref{eqn:260}).
    
    Now we assume $s = 1$.
    We first rewrite the problem as $E_\delta = \inf_{\nu \in [0, \delta]} \tilde{E}_\nu$ with
    \begin{equation*}
        \tilde{E}_\nu = \inf\{E(q): q \in H^1_{\loc}, \inf_{x \in \R} |q(x)| = \nu\} \,, \qquad \nu \in [0,\delta]\,.
    \end{equation*}
    Of course we expect that $\tilde{E}_\nu$ is decreasing in $\nu$ and hence $E_\delta = \tilde{E}_\delta$. This will be verified once we have calculated $\tilde{E}_{\nu}$.
    Using invariance under translations, phase rotations, and mirror symmetry, we can equivalently consider the minimization problem
    \begin{equation*}
        \tilde{E}_\nu = 2 \inf \left\{ \frac12 \int_0^\infty |\del_x q|^2 + (|q|^2 - 1)^2 \dd x : q \in H^1_{\loc}(\R_\geq), q(0) = \nu \right\} \,.
    \end{equation*}
    We now follow the same arguments as in \cite[Lemma 1]{BethuelGravejatSaut} to find a minimizer. Consider a minimizing sequence $(q_n)_{n \in \N}$. As $E^s(q_n)$ is uniformly bounded, so is $\|\del_x q_n\|_{L^2(\R_\geq)}$. The Banach-Alaoglu theorem then implies, up to a subsequence, that $\del_x q_n \longrightarrow p_\nu'$ for some $p_\nu' \in L^2(\R_\geq)$. Furthermore as $q_n(0) = \nu$ is fixed, we have a Poincare inequality $\|q_n\|_{W^{1,2}(B_0)} \leq C(s, B)\, \|\del_x q_n\|_{L^2(B_0)} $ on any finite interval $B_0 \subset \R_\geq$. Then we can use compactness of the Sobolev embedding $H^1 \xhookrightarrow{\quad} L^\infty$ to find, up to a subsequence, that $q_n \longrightarrow p_\nu$ in $L^\infty_{\loc}(\R_\geq)$ for some $p_\nu \in H^1_{\loc}(\R_\geq)$, with $p_\nu'$ indeed being its distributional derivative. Now we can conclude with Fatou's lemma that $p_\nu$ is a minimizer for $\tilde{E}_{\nu}$:
    \begin{align*}
         \int_0^\infty (|p_0|^2 - 1)^2 + |p_0'|^2 \dd x &=  \int_0^\infty \liminf_n (|q_n|^2 - 1)^2 + \liminf_n |q_n'|^2 \dd x \\
        &\leq \liminf_n  \int_0^\infty (|q_n|^2 - 1)^2 + \liminf_n |q_n'|^2 \dd x \\
        &= E_\nu\,.
    \end{align*}
    For the case $\nu = 0$, we obtain the Euler-Lagrange equation
    \begin{equation*}
        p_0'' - 2 p_0 (1 - |p_0|^2) = 0\,.
    \end{equation*}
    Then as $p_0(0) = 0$ and $E(p_0) < \infty$, \cite[Theorem 1]{BethuelGravejatSaut} implies that $p_0 = \tanh$ is the unique solution. Consequently, it must be the case that for $a > 0$ the function $p_\nu(x) = \tanh(x + a)$ is a minimizer for the problem with $\nu = \tanh(a)$, as otherwise one could modify $p_0$ on $[r,\infty)$ to find an admissible function with strictly smaller energy for the minimization problem of $\tilde{E}_0$. This implies that the $q_\delta$ defined in (\ref{eqn:27}) are minimizers for $\tilde{E}_\nu$.
    
    With $a = \tanh^{-1}(\nu)$, and noting the identities
    \begin{equation*}
        \cosh(a) = \frac{1}{\sqrt{1 - \nu^2}}, \qquad \cosh(2 a) = - \frac{1 + \nu^2}{1 - \nu^2}, \qquad \sech^2(a) = 1 - \nu^2 \,,
    \end{equation*}
    we compute
    \begin{align*}
        E(q_\nu) &= 2 \cdot \frac12 \int_a^\infty (\tanh(x)')^2 + (\tanh(x)^2 - 1)^2 \dd x \\
        &= \int_a^\infty \sech^4(x) + \sech^4(x) \dd x \,.
    \end{align*}
    Evaluating the integral yields
    \begin{align*}
        E(q_\nu) &= \left[ \frac23 (\cosh(2x) + 2) \tanh(x) \sech^2(x) \right]_r^\infty \\
        &= \frac23 \Big(2 - \nu (1 - \nu^2) \Big(2 + \frac{1 + \nu^2}{1 - \nu^2}\Big) \Big) \\
        &= \frac43 - 2 \nu + \frac23 \nu^3 \,.
    \end{align*}
\end{proof}

\section{Littlewood-Paley theory and proof of Lemma \ref{lem:2}}

The proof uses the Bony decomposition
\begin{equation*}
    f g = T_f g + R(f, g) + T_g f \,,
\end{equation*}
which J.M. Bony introduced in his 1981 paper \cite{Bony}. It relies on the Littlewood-Paley theory, for which we refer the reader to \cite[Chp. 2]{BCD}. We give a brief introduction below, always only considering the one-dimensional case.

Let $\phi \in C_c^\infty(\{\xi: \frac34 < |\xi| < \frac83\}$ and $\chi \in C_c^\infty(\{\xi: |\xi| < \frac43\})$ be non-negative functions on $\R$ so that
\begin{equation*}
    \chi(\xi) + \sum_{j=0}^\infty \phi(2^{-j}) = 1 \,.
\end{equation*}
This is called a dyadic partition of unity. We define for $j \in \Z$ the operators
\begin{align*}
    \Delta_j: \mathcal{S}' &\longrightarrow \mathcal{S}' \\
    f &\longmapsto \Delta_j f
    = \begin{cases}
        \phi(2^{-j} \cdot) \widehat{f} &,j \geq 0 \\
        \chi \widehat{f} &,j = -1 \\
        0 &,j \leq -2
    \end{cases}
\end{align*}
and $S_j = \sum_{j' < j} \Delta_{j'}$. These operators have nice properties such as $\|S_j f\|_{L^p} \leq C(p) \|f\|_{L^p}$, $p \in [1, \infty]$.
At least formally, we have the decomposition
\begin{equation*}
    \Id = \lim_{j \rightarrow \infty} S_j = \sum_j \Delta_j \,.
\end{equation*}
The Bony decomposition is given by
\begin{equation*}
    f g = \sum_{j, k} \Delta_j f \Delta_k g = T_f g + R(f, g) + T_g f \,,
\end{equation*}
where we define
\begin{equation*}
    T_f g = \sum_j S_{j-1} f \Delta_j g \qquad \qquad R(f, g) = \sum_j \sum_{|\nu| \leq 1} \Delta_{j + \nu} f \Delta_j g \,.
\end{equation*}

In the Littlewood-Paley setting it is easy to define the Besov spaces $B_{p,q}^s$ for $1 \leq p, r \leq \infty$, $s \in \R$ by
\begin{equation*}
    B_{p,q}^s = \big\{f \in \mathcal{S}'(\R;\C): \|f\|_{B_{p,q}^s} < \infty \big\} \,,
\end{equation*}
where
\begin{equation*}
    \|f\|_{B_{p,q}^s} = \big\| \big(2^{j s} \|\Delta_j f\|_{L^p}\big)_{j \in \Z} \big\|_{\ell^r} \,.
\end{equation*}
It is evident that
\begin{equation*}
    B_{2,2}^s = \{f \in \mathcal{S}'(\R;\C): \|\langle \xi \rangle^s \widehat{f}\|_{L^2(\R)} < \infty\} = H^s(\R;\C)\,.
\end{equation*}

\begin{proof}[Proof of Lemma \ref{lem:2}]
    We only prove (\ref{eqn:14}) as the proof of (\ref{eqn:13}) is analogous and strictly simpler.
    Consider the decomposition $f g = g S_0 f + g (1 - S_0) f$.
    Since $S_0 f$ is spectrally supported in a fixed ball there exists a constant $N \in \N$ so that $\Delta_k(S_0 f \Delta_j g) = 0$ unless $|k - j| \leq N$. Consequently,
    \begin{align*}
        \|g S_0 f\|_{H^{s-1}}^2 &= \sum_{j \in \Z} 2^{2j(s-1)} \left\|\sum_{|\nu| \leq N} \Delta_j(S_0 f \Delta_{j + \nu} g)\right\|_{L^2}^2 \\
        &\leq C(N) \|S_0 f\|_{L^\infty} ^2 \sum_{j \in \Z} 2^{2j(s-1)}\|\Delta_k g\|_{L^2}^2 \\
        &= C(N) \|S_0 f\|_{L^\infty}^2 \|g\|_{H^{s-1}}^2 \,.
    \end{align*}
    
    It remains to estimate $\|g (1 - S_0) f\|_{H^{s-1}}$. To simplify notation, we now write $f$ for $(1 - S_0) f$ and derive an estimate by $\|f\|_{H^s}$.
    Note that $S_{j-1} f \Delta_j g$ is only non-zero if $j \geq 1$, and in that case it is a convolution of a ball with an annulus of much larger radius. As a result, there exists an annulus $\mathcal{C}$ so that $\mathcal{F}[S_{j-1} f \Delta_j g]$ is supported in $2^j \mathcal{C}$, and so \cite[Lemma 2.69]{BCD} implies
    \begin{equation*}
        \|T_f g\|_{H^{s-1}} \lesssim \big\| 2^{j(s-1)} \|S_{j-1} f \Delta_j g\|_{L^2} \big\|_{\ell^2(\Z)} \,.
    \end{equation*}
    Since
    \begin{equation*}
        \|S_{j-1} f \Delta_j g\|_{L^2} \leq \|S_{j-1} f\|_{L^\infty} \|\Delta_j g\|_{L^2} \leq \|f\|_{L^\infty} \|\Delta_j g\|_{L^2} \,,
    \end{equation*}
    this implies 
    \begin{equation*}
        \|T_f g\|_{H^{s-1}} \lesssim \|g\|_{H^{s-1}} \|f\|_{L^\infty} \,.
    \end{equation*}
    For the same reason as before, we have
    \begin{equation*}
        \|T_g f\|_{H^{s-1}} \lesssim \big\| 2^{j(s-1)} \|S_{j-1} g \Delta_j f\|_{L^2} \big\|_{\ell^2} \,.
    \end{equation*}
    Here we consider two cases. 
    If $s \leq 1$ then we use the Bernstein inequality \cite[Lemma 2.1]{BCD}. It states that
    \begin{equation*}
        \supp \widehat{u} \subset \lambda B \Longrightarrow \|u\|_{L^\infty} \leq C(B) \lambda^{\frac12} \|u\|_{L^2}
    \end{equation*}
    for any fixed ball $B$. This yields
    \begin{align*}
        2^{j(s-1)} \|S_{j-1} g \Delta_j f\|_{L^2} &\leq 2^{j(s-1)} \|S_{j-1} g\|_{L^2} \|\Delta_j f\|_{L^\infty} \\
        &\lesssim \|S_{j-1} g\|_{H^{s-1}} 2^{\frac{j}{2}} \|\Delta_j f\|_{L^2} \\
        &\lesssim \|g\|_{H^{s-1}} 2^{\frac{j}{2}} \|\Delta_j f\|_{L^2} \,.
    \end{align*}
    Here we have used
    \begin{equation*}
        2^{2j(s-1)} \|S_{j-1} g\|_{L^2}^2 = \sum_{j' < j - 1} \underset{\leq 1}{\underbrace{2^{2(j - j')(s-1)}}}  2^{2j'(s-1)} \|\Delta_{j'} g\|_{L^2}^2 \leq \|S_{j-1} g\|_{H^{s-1}}^2 \leq \|g\|_{H^{s-1}}^2 \,.
    \end{equation*}
    We see that
    \begin{equation*}
        \|T_g f\|_{H^{s-1}} \lesssim \|g\|_{H^{s-1}} \|f\|_{H^{\frac12}} \,.
    \end{equation*}
    For the case $s > 1$ we estimate
    \begin{align*}
        2^{j(s-1)} \|S_{j-1} g \Delta_j f\|_{L^2} &\leq 2^{j(s-1)} \|S_{j-1} g\|_{L^\infty} \|\Delta_j f\|_{L^2} \\
        &\lesssim \|2^{-j} S_{j-1} g\|_{H^1} 2^{js} \|\Delta_j f\|_{L^2} \\
        &\lesssim \|g\|_{L^2} 2^{js} \|\Delta_j f\|_{L^2} \\
    \end{align*}
    and obtain
    \begin{equation*}
        \|T_g f\|_{H^{s-1}} \lesssim \|g\|_{H^{s-1}} \|f\|_{H^s} \,.
    \end{equation*}
    
    It remains to estimate the remainder terms $R(f, g)$. Here let it be noted that there exists an integer $N > 0$, independent of $j$, so that $\sum_{|\nu| \leq 1} \Delta_{j-\nu} f \Delta_j g$ is spectrally supported in a ball of radius $2^{j + N - 1}$. In this case we know by \cite[Lemma 2.84]{BCD} that
    \begin{equation} \label{eqn:28}
        \|R(f, g)\|_{B_{p,r}^{\tilde{s}}} \leq C(p,r,\tilde{s}) \Big\| 2^{j \tilde{s}} \big\| \sum_{|\nu| = 1} \Delta_{j - \nu} g \Delta_j f \big\|_{L^p} \Big\|_{\ell^r(\Z)}
    \end{equation}
    for $\tilde{s} > 0$. This does not work in general if $\tilde{s} < 0$.
    Therefore, we use the embedding 
    \begin{equation*}
        \|R(f, g)\|_{H^{s-1}} \leq C(s) \|R(f, g)\|_{B_{1,1}^{s-\frac12}}
    \end{equation*}
    in order to apply (\ref{eqn:28}) with $\tilde{s} = s - \frac12 > 0$ and $p, r = 1$. 
    Now we can conclude via H\"older's and Young's inequalities for sequences:
    \begin{align*}
        \allowdisplaybreaks
        \Big\| 2^{j(s-\frac12)} \Big\| \sum_{|\nu| = 1} \Delta_{j - \nu} f \Delta_j g \Big\|_{L^1} \Big\|_{\ell^1(\Z)} 
        &\lesssim
        \sum_{|\nu| \leq 1} 2^{\frac{\nu}{2}} \big\| 2^{\frac{j - \nu }{2}} \|\Delta_{j - \nu} f\|_{L^2} \big\|_{\ell^2(\Z)} \big\| 2^{j(s-1)} \|\Delta_j g\|_{L^2} \big\|_{\ell^2(\Z)} \\
        &\lesssim \|g\|_{H^{s-1}} \|f\|_{H^{\frac12}} \,.
    \end{align*}
    
\end{proof}

\section{Uniqueness for the Gross-Pitaevskii equation}

\begin{proof}[Proof of Lemma \ref{lem:71}]
    Recall that $q_1, q_2 \in C(I; L^2_{\loc}) \cap L^\infty(I; L^\infty \cap \dot{H}^1)$ are two distributional solutions of (GP) on an open interval $I \ni 0$ with the same initial data $q_0$. It follows that $q_1, q_2 \in L^\infty(I; W^{1,2}(B))$ for any arbitrary ball $B \subset \R$. We define $b = q_1 - q_2$ and compute that it solves in distribution the following equation:
    \begin{align*}
        i \del_t b + \del_{xx} b &= 2 q_1 (|q_1|^2 - 1) - 2 q_2 ( |q_2|^2 - 1) \\
        &= 2 b (|q_1|^2 - 1) + 2 b (|q_2|^2 - 1) + 2 q_2 (|q_1|^2 - 1) - 2 q_1 (|q_2|^2 - 1) \\
        &= 2 b (|q_1|^2 + |q_2|^2 - 2 + 1) + 2 (q_2 |q_1|^2 - q_1 |q_2|^2) \\
        &= 2b ((b + q_2) \overline{(b + q_2)} + |q_2|^2 - 1) + 2 (q_2 |b + q_2|^2 - (b + q_2) |q_2|^2) \\
        &= 2b (|b|^2 + b \overline{q_2} + \overline{b} q_2 + 2 |q_2|^2 - 1) \\
        &+ 2 (q_2 |b|^2 + b |q_2|^2 + \overline{b} q_2^2 + q_2 |q_2|^2 - b |q_2|^2 - q_2 |q_2|^2) \\
        &= 2b (|b|^2 + b \overline{q_2} + \overline{b} q_2 + 2 |q_2|^2 - 1)
        + 2 (q_2 |b|^2 + \overline{b} q_2^2) \\
        &= 2 |b|^2 b + 4 |b|^2 q_2 + 2 b^2 \overline{q_2} + 2 b (2 |q_2|^2 - 1) + 2 \overline{b} q_2^2 \,.
    \end{align*}
    We know that $\del_{xx} b \in L^\infty(I; W^{-1,2}(B))$. Then $b$ solving the equation implies $\del_t b \in L^\infty(I; W^{-1,2}(B))$. Using duality and the algebra property of $W^{1,2}(B)$, we find that also $\overline{b} \, \del_t b, \del_t (|b|^2) \in L^\infty(I; W^{-1,2}(B))$.
    
    Let $\phi_n(x) = \phi(\frac{x}{n})$ where $\phi \in C_c^\infty([-2, 2]; [0,1])$ and $\phi\big\vert_{[-1,1]} = 1$. One may choose $\phi$ in such a way that there exists $K > 0$ with $|\del_x \phi| \leq K \sqrt{\phi}$ and in particular $|\del_x \phi_n| \leq K n^{-1} \sqrt{\phi_n}$. We test the above with $\overline{b} \phi_n$ and take the imaginary part. On the left-hand side, we have
    \begin{align*}
        \int_{\R} \Imag[i (\del_t b) \overline{b}\phi_n] + \Imag[(\del_{xx} b) \overline{b} \phi_n] \dd x &= \int_{\R} \frac12 \phi_n \del_t (|b|^2) - \Imag[ (\del_x b) \overline{b} \del_x \phi_n] \dd x \,.
    \end{align*}
    Therefore for a fixed time $t \in I$,
    \begin{align*}
        \frac12 \frac{d}{dt} \int_{\R} \phi_n |b|^2 \dd x &= \int_{\R} \Imag[(\del_x b) \overline{b} \del_x \phi_n] \dd x \\
        &+ \int_{\R} \big(2 |b|^4 + 4 |b|^2 \overline{b} q_2 + 2 |b|^2 b \overline{q_2} + 2 |b|^2 (2 |q_2|^2 - 1) + 2 \overline{b}^2 q_2^2 \big) \phi_n \dd x \\
        &= (I) + (II) \,.
    \end{align*}
    We estimate 
    \begin{equation*}
        (I) \leq \|\del_x b\|_{L^2_x} \|b \del_x \phi_n\|_{L^2_x} \leq \big(\|q_1\|_{L^\infty_t \dot{H}^1_x} + \|q_2\|_{L^\infty_t \dot{H}^1_x}\big) K n^{-1} \big\|b \sqrt{\phi_n}\big\|_{L^2_x}
    \end{equation*}
    and
    \begin{align*}
        (II) &\leq C \big\|b \sqrt{\phi_n}\|_{L^2_x}^2 \big\| \big\| |b|^2 + |b| |q_2| + |q_2|^2 + 1 \big\|_{L^\infty_{t,x}} \\
        &\leq C \big\|b \sqrt{\phi_n}\big\|_{L^2_x}^2 \big( 1 + \|q_1\|_{L^\infty_{t,x}}^2 + \|q_2\|_{L^\infty_{t,x}}^2 \big) \,.
    \end{align*}
    We have shown that there exists some $C > 0$, depending on $q_1, q_2$ but independent of time, such that
    \begin{equation*}
        \frac12 \frac{d}{dt} \big( \|b \sqrt{\phi_n}\|_{L^2_x}^2 \big) \leq C \left(\frac{1}{\sqrt{n}} \|b \sqrt{\phi_n}\|_{L^2_x} + \|b \sqrt{\phi_n}\|_{L^2_x}^2 \right) \,.
    \end{equation*}
    In particular
    \begin{equation*}
        \frac{d}{dt} \|b \sqrt{\phi_n}\|_{L^2_x} \leq C \left( \frac{1}{\sqrt{n}} + \|b \sqrt{\phi_n}\|_{L^2_x} \right) \,.
    \end{equation*}
    Now Gr\"onwall's inequality implies for any fixed $t > 0$ that
    \begin{equation*}
        \|b(t)\|_{L^2_x} \xleftarrow{n \rightarrow \infty} \|b(t) \sqrt{\phi_n}\|_{L^2_x} \leq \Bigg(\underset{ = 0}{\underbrace{\|b(0) \sqrt{\phi_n}\|_{L^2_x} }}+ C \frac{t}{\sqrt{n}} \Bigg) e^{C t} \xrightarrow{n \rightarrow \infty} 0\,,
    \end{equation*}
    hence $q_1 = q_2$ for positive times. The argument for negative times is analogous.
\end{proof}

\printbibliography

@article {KochLiao,
    AUTHOR = {H. Koch and X. Liao},
     TITLE = {Conserved energies for the one dimensional
              {G}ross-{P}itaevskii equation},
   JOURNAL = {Adv. Math.},
  FJOURNAL = {Advances in Mathematics},
    VOLUME = {377},
      YEAR = {2021},
     PAGES = {Paper No. 107467, 83},
      ISSN = {0001-8708},
   MRCLASS = {35Q55 (37K10)},
  MRNUMBER = {4186010},
MRREVIEWER = {W.-H. Steeb},
       DOI = {10.1016/j.aim.2020.107467},
       URL = {https://doi.org/10.1016/j.aim.2020.107467},
}

@article {Mohamad,
    AUTHOR = {H. Mohamad},
     TITLE = {Hydrodynamical form for the one-dimensional {G}ross-{P}itaevskii equation},
   JOURNAL = {Electron. J. Differential Equations},
  FJOURNAL = {Electronic Journal of Differential Equations},
      YEAR = {2014},
     PAGES = {No. 141, 27},
   MRCLASS = {35Q55 (35B30 35C07 35C08)},
  MRNUMBER = {3239384},
}

@article{Audiard,
    AUTHOR = {C. Audiard},
     TITLE = {Global well-posedness of a system from quantum hydrodynamics
              for small data},
   JOURNAL = {Confluentes Math.},
  FJOURNAL = {Confluentes Mathematici},
    VOLUME = {7},
      YEAR = {2015},
    NUMBER = {2},
     PAGES = {7--16},
   MRCLASS = {35Q35 (35B30 76Y05)},
  MRNUMBER = {3466437},
MRREVIEWER = {Luc Paquet},
       DOI = {10.5802/cml.21},
       URL = {https://doi.org/10.5802/cml.21},
}

@article{AntonelliMarcati,
    AUTHOR = {P. Antonelli and P. Marcati},
     TITLE = {On the finite energy weak solutions to a system in quantum fluid dynamics},
   JOURNAL = {Comm. Math. Phys.},
  FJOURNAL = {Communications in Mathematical Physics},
    VOLUME = {287},
      YEAR = {2009},
    NUMBER = {2},
     PAGES = {657--686},
      ISSN = {0010-3616},
   MRCLASS = {82D50 (35D30 35Q35 76Y05 82D37 82D55)},
  MRNUMBER = {2481754},
       DOI = {10.1007/s00220-008-0632-0},
       URL = {https://doi.org/10.1007/s00220-008-0632-0},
}

@article {AudiardHaspot,
    AUTHOR = {C. Audiard and B. Haspot},
     TITLE = {Global well-posedness of the {E}uler-{K}orteweg system for
              small irrotational data},
   JOURNAL = {Comm. Math. Phys.},
  FJOURNAL = {Communications in Mathematical Physics},
    VOLUME = {351},
      YEAR = {2017},
    NUMBER = {1},
     PAGES = {201--247},
      ISSN = {0010-3616},
   MRCLASS = {35Q35 (35B30 35Q31 35Q40 35Q53 35Q55 76N99 81Q80)},
  MRNUMBER = {3613503},
MRREVIEWER = {Paolo Secchi},
       DOI = {10.1007/s00220-017-2843-8},
       URL = {https://doi.org/10.1007/s00220-017-2843-8},
}

@article {Audiard2021,
    AUTHOR = {C. Audiard},
     TITLE = {On the time of existence of solutions of the
              {E}uler-{K}orteweg system},
   JOURNAL = {Ann. Fac. Sci. Toulouse Math. (6)},
  FJOURNAL = {Annales de la Facult\'{e} des Sciences de Toulouse. Math\'{e}matiques.
              S\'{e}rie 6},
    VOLUME = {30},
      YEAR = {2021},
    NUMBER = {5},
     PAGES = {1139--1183},
      ISSN = {0240-2963},
   MRCLASS = {35Q31},
  MRNUMBER = {4401388},
MRREVIEWER = {Wei Luo},
       DOI = {10.5802/afst.1696},
       URL = {https://doi.org/10.5802/afst.1696},
}

@book{BCD,
    AUTHOR = {H. Bahouri and J.-Y. Chemin and R. Danchin},
     TITLE = {Fourier analysis and nonlinear partial differential equations},
    SERIES = {Grundlehren der mathematischen Wissenschaften [Fundamental
              Principles of Mathematical Sciences]},
    VOLUME = {343},
 PUBLISHER = {Springer, Heidelberg},
      YEAR = {2011},
     PAGES = {xvi+523},
      ISBN = {978-3-642-16829-1},
   MRCLASS = {35-02 (35L72 35Q30 42-02 42B37 76B03 76D03 76N10)},
  MRNUMBER = {2768550},
MRREVIEWER = {Peter R. Massopust},
       DOI = {10.1007/978-3-642-16830-7},
       URL = {https://doi.org/10.1007/978-3-642-16830-7},
}

@book{McLean,
    AUTHOR = {W. McLean},
     TITLE = {Strongly elliptic systems and boundary integral equations},
 PUBLISHER = {Cambridge University Press, Cambridge},
      YEAR = {2000},
     PAGES = {xiv+357},
      ISBN = {0-521-66332-6},
   MRCLASS = {35J45 (47F05 47G10 47N20 65N38)},
  MRNUMBER = {1742312},
MRREVIEWER = {Dorina I. Mitrea},
}

@article{Bony,
    AUTHOR = {J.-M. Bony},
     TITLE = {Calcul symbolique et propagation des singularit\'{e}s pour les
              \'{e}quations aux d\'{e}riv\'{e}es partielles non lin\'{e}aires},
   JOURNAL = {Ann. Sci. \'{E}cole Norm. Sup. (4)},
  FJOURNAL = {Annales Scientifiques de l'\'{E}cole Normale Sup\'{e}rieure. Quatri\`eme
              S\'{e}rie},
    VOLUME = {14},
      YEAR = {1981},
    NUMBER = {2},
     PAGES = {209--246},
      ISSN = {0012-9593},
   MRCLASS = {35S99 (35G20)},
  MRNUMBER = {631751},
MRREVIEWER = {J. Lacroix},
       URL = {http://www.numdam.org/item?id=ASENS_1981_4_14_2_209_0},
}

@article{Gross,
    AUTHOR = {E.P. Gross},
     TITLE = {Hydrodynamics of a superfluid condensate},
   JOURNAL = {J. Math. Phys.},
    VOLUME = {4},
    NUMBER = {2},
     PAGES = {195-207},
      YEAR = {1963},
}

@article{Pitaevskii,
    AUTHOR = {L.P. Pitaevskii},
     TITLE = {Vortex lines in an imperfect Bose gas},
   JOURNAL = {Sov. Phys. JETP},
    VOLUME = {13},
      YEAR = {1961},
     PAGES = {451–454}
}

@article{ErdosSchleinYau,
    AUTHOR = {L. Erdős and B. Schlein and H.-T. Yau},
     TITLE = {Derivation of the {G}ross-{P}itaevskii hierarchy for the
              dynamics of {B}ose-{E}instein condensate},
   JOURNAL = {Comm. Pure Appl. Math.},
  FJOURNAL = {Communications on Pure and Applied Mathematics},
    VOLUME = {59},
      YEAR = {2006},
    NUMBER = {12},
     PAGES = {1659--1741},
      ISSN = {0010-3640},
   MRCLASS = {82C10 (35Q55 47N55)},
  MRNUMBER = {2257859},
MRREVIEWER = {Barbara Prinari},
       DOI = {10.1002/cpa.20123},
       URL = {https://doi.org/10.1002/cpa.20123},
}

@article{Zhidkov1987,
    AUTHOR = {P.E. Zhidkov}, 
     TITLE = {The Cauchy problem for the nonlinear Schr\"odinger equation},
   JOURNAL = {Joint Inst. Nuclear Res., Dubna},
      YEAR = {1987},
     PAGES = {15}
}

@book{Zhidkov2001,
    AUTHOR = {P.E. Zhidkov},
     TITLE = {Korteweg-de {V}ries and nonlinear {S}chr\"{o}dinger equations:
              qualitative theory},
    SERIES = {Lecture Notes in Mathematics},
    VOLUME = {1756},
 PUBLISHER = {Springer-Verlag, Berlin},
      YEAR = {2001},
     PAGES = {vi+147},
      ISBN = {3-540-41833-4},
   MRCLASS = {35Q53 (35B30 35B35 35Q55 37K15 37K40)},
  MRNUMBER = {1831831},
MRREVIEWER = {Woodford W. Zachary},
}

@article{Gallo,
    AUTHOR = {C. Gallo},
     TITLE = {Schr\"{o}dinger group on {Z}hidkov spaces},
   JOURNAL = {Adv. Differential Equations},
  FJOURNAL = {Advances in Differential Equations},
    VOLUME = {9},
      YEAR = {2004},
    NUMBER = {5-6},
     PAGES = {509--538},
      ISSN = {1079-9389},
   MRCLASS = {35Q55 (35A30 47H20)},
  MRNUMBER = {2099970},
MRREVIEWER = {Peter E. Zhidkov},
}

@article{Gerard2006,
    AUTHOR = {P. G\'{e}rard},
     TITLE = {The {C}auchy problem for the {G}ross-{P}itaevskii equation},
   JOURNAL = {Ann. Inst. H. Poincar\'{e} C Anal. Non Lin\'{e}aire},
  FJOURNAL = {Annales de l'Institut Henri Poincar\'{e} C. Analyse Non Lin\'{e}aire},
    VOLUME = {23},
      YEAR = {2006},
    NUMBER = {5},
     PAGES = {765--779},
      ISSN = {0294-1449},
   MRCLASS = {35Q55 (35A15 35B30 35J20 81Q20)},
  MRNUMBER = {2259616},
MRREVIEWER = {Zhi-Qiang Wang},
       DOI = {10.1016/j.anihpc.2005.09.004},
       URL = {https://doi.org/10.1016/j.anihpc.2005.09.004},
}

@article{RomainTadahiroOanaMonica,
    AUTHOR = {R. Killip and T. Oh and O. Pocovnicu and M. Vi\c{s}an},
     TITLE = {Global well-posedness of the {G}ross-{P}itaevskii and
              cubic-quintic nonlinear {S}chr\"{o}dinger equations with
              non-vanishing boundary conditions},
   JOURNAL = {Math. Res. Lett.},
  FJOURNAL = {Mathematical Research Letters},
    VOLUME = {19},
      YEAR = {2012},
    NUMBER = {5},
     PAGES = {969--986},
      ISSN = {1073-2780},
   MRCLASS = {35Q55 (35B30)},
  MRNUMBER = {3039823},
MRREVIEWER = {Gianmaria Verzini},
       DOI = {10.4310/MRL.2012.v19.n5.a1},
       URL = {https://doi.org/10.4310/MRL.2012.v19.n5.a1},
}

@article {CarlesDanchinSaut,
    AUTHOR = {R. Carles and R. Danchin and J.-C. Saut},
     TITLE = {Madelung, {G}ross-{P}itaevskii and {K}orteweg},
   JOURNAL = {Nonlinearity},
  FJOURNAL = {Nonlinearity},
    VOLUME = {25},
      YEAR = {2012},
    NUMBER = {10},
     PAGES = {2843--2873},
      ISSN = {0951-7715},
   MRCLASS = {35Q55 (35-02 76Y05)},
  MRNUMBER = {2979973},
       DOI = {10.1088/0951-7715/25/10/2843},
       URL = {https://doi.org/10.1088/0951-7715/25/10/2843},
}

@inproceedings{BethuelGravejatSaut,
    AUTHOR = {F. Béthuel and P. Gravejat and J.-C. Saut},
     TITLE = {Existence and properties of travelling waves for the
              {G}ross-{P}itaevskii equation},
 BOOKTITLE = {Stationary and time dependent {G}ross-{P}itaevskii equations},
    SERIES = {Contemp. Math.},
    VOLUME = {473},
     PAGES = {55--103},
 PUBLISHER = {Amer. Math. Soc., Providence, RI},
      YEAR = {2008},
   MRCLASS = {35Q55 (35C07 35Q51)},
  MRNUMBER = {2522014},
MRREVIEWER = {Ayman Kachmar},
       DOI = {10.1090/conm/473/09224},
       URL = {https://doi.org/10.1090/conm/473/09224},
}

@article{KillipVisan,
    AUTHOR = {R. Killip and M. Vișan},
     TITLE = {Kd{V} is well-posed in {$H^{-1}$}},
   JOURNAL = {Ann. of Math. (2)},
  FJOURNAL = {Annals of Mathematics. Second Series},
    VOLUME = {190},
      YEAR = {2019},
    NUMBER = {1},
     PAGES = {249--305},
      ISSN = {0003-486X},
   MRCLASS = {35Q53 (35B30 37K10)},
  MRNUMBER = {3990604},
MRREVIEWER = {John Albert},
       DOI = {10.4007/annals.2019.190.1.4},
       URL = {https://doi.org/10.4007/annals.2019.190.1.4},
}

@article{Gerard2008,
    AUTHOR = {P. G\'{e}rard},
     TITLE = {The {G}ross-{P}itaevskii equation in the energy space},
 BOOKTITLE = {Stationary and time dependent {G}ross-{P}itaevskii equations},
    SERIES = {Contemp. Math.},
    VOLUME = {473},
     PAGES = {129--148},
 PUBLISHER = {Amer. Math. Soc., Providence, RI},
      YEAR = {2008},
   MRCLASS = {35Q55 (35J20)},
  MRNUMBER = {2522016},
MRREVIEWER = {Ayman Kachmar},
       DOI = {10.1090/conm/473/09226},
       URL = {https://doi.org/10.1090/conm/473/09226},
}

@article{KochLiao2022,
    AUTHOR = {H. Koch and X. Liao},
     TITLE = {Conserved energies for the one dimensional
              {G}ross-{P}itaevskii equation: low regularity case},
   JOURNAL = {Adv. Math.},
    VOLUME = {420},
     PAGES = {Paper No. 108996, 61},
      YEAR = {2023},
      ISSN = {0001-8708,1090-2082},
   MRCLASS = {35Q55 (37K10)},
  MRNUMBER = {4568059},
       DOI = {10.48550/ARXIV.2204.06293},
}

@article {DanchinLiao,
    AUTHOR = {R. Danchin and X. Liao},
     TITLE = {On the well-posedness of the full low {M}ach number limit
              system in general critical {B}esov spaces},
   JOURNAL = {Commun. Contemp. Math.},
  FJOURNAL = {Communications in Contemporary Mathematics},
    VOLUME = {14},
      YEAR = {2012},
    NUMBER = {3},
     PAGES = {1250022, 47},
      ISSN = {0219-1997},
   MRCLASS = {35Q35 (35B30)},
  MRNUMBER = {2943836},
       DOI = {10.1142/S0219199712500228},
       URL = {https://doi.org/10.1142/S0219199712500228},
}

@article{Landau,
    AUTHOR = {L. Landau},
     TITLE = {Theory of the superfluidity of Helium II},
   JOURNAL = {Phys. Rev.},
    VOLUME = {60},
     ISSUE = {4},
     PAGES = {356--358},
      YEAR = {1941},
 PUBLISHER = {American Physical Society},
       DOI = {10.1103/PhysRev.60.356},
       URL = {https://link.aps.org/doi/10.1103/PhysRev.60.356},
}

@article{Loffredo1993,
    AUTHOR = {M.I. Loffredo, L.M. and Morato},
     TITLE = {On the creation of quantized vortex lines in rotating He II},
   JOURNAL = {Il Nuovo Cimento B (1971-1996)},
    VOLUME = {108},
    NUMBER = {2},
     PAGES = {205--215},
      YEAR = {1993},
 PUBLISHER = {Springer},
       DOI = {10.1007/BF02874411},
}

@article{Dalfovo1998,  
    AUTHOR = {D. Franco and S. Giorgini and L.P. Pitaevskii and S. Stringari},
     TITLE = {Theory of Bose-Einstein condensation in trapped gases},
   JOURNAL = {Rev. Mod. Phys.},
  PUBLISHER = {American Physical Society},
      YEAR = {1999},
    VOLUME = {71},
     ISSUE = {3},
     PAGES = {463--512},
       DOI = {10.1103/RevModPhys.71.463},
       URL = {https://link.aps.org/doi/10.1103/RevModPhys.71.463}
}

@article{Grant1973,
    AUTHOR = {J. Grant},
     TITLE = {Pressure and stress tensor expressions in the fluid mechanical formulation of the Bose condensate equations},
   JOURNAL = {J. phys., A Math. nucl. gen.},
       URL = {https://dx.doi.org/10.1088/0305-4470/6/11/001},
      YEAR = {1973},
    VOLUME = {6},
    NUMBER = {11},
     PAGES = {L151},
       DOI = {10.1088/0305-4470/6/11/001},
       URL = {https://dx.doi.org/10.1088/0305-4470/6/11/001},
}

@article{Feynman,
    AUTHOR = {R.P. Feynman},
     TITLE = {Superfluidity and superconductivity},
   JOURNAL = {Rev. Mod. Phys.},
    VOLUME = {29},
     ISSUE = {2},
     PAGES = {205--212},
      YEAR = {1957},
 PUBLISHER = {American Physical Society},
       URL = {https://link.aps.org/doi/10.1103/RevModPhys.29.205},
       DOI = {10.1103/RevModPhys.29.205},
}

@article{Gardner,
    AUTHOR = {C.L. Gardner},
     TITLE = {The quantum hydrodynamic model for semiconductor devices},
   JOURNAL = {SIAM J. Appl. Math.},
      ISSN = {00361399},
       URL = {http://www.jstor.org/stable/2102225},
    NUMBER = {2},
     PAGES = {409--427},
 PUBLISHER = {Society for Industrial and Applied Mathematics},
   URLDATE = {2023-07-26},
    VOLUME = {54},
      YEAR = {1994},
       DOI = {10.1137/S0036139992240425},
}

@inproceedings{AHMZ,
    AUTHOR = {P. Antonelli and L.E. Hientzsch and P. Marcati and H. Zheng},
  CROSSREF = {RIMS},
     PAGES = {107--129},
     TITLE = {On some results for quantum hydrodynamical models},
       URL = {http://hdl.handle.net/2433/241991},
    VOLUME = {2070},
      YEAR = {2018},
}

@proceedings{RIMS,
    EDITOR = {T. Kobayashi},
 BOOKTITLE = {Mathematical analysis in fluid and gas dynamics},
      ISSN = {1880-2818},
 PUBLISHER = {RIMS Kokyuroku},
     TITLE = {Mathematical Analysis in Fluid and Gas Dynamics},
      YEAR = {2018},
}

@article {Bresch2017,
    AUTHOR = {D. Bresch and M. Gisclon and I. Lacroix-Violet},
     TITLE = {On {N}avier-{S}tokes-{K}orteweg and {E}uler-{K}orteweg
              systems: application to quantum fluids models},
   JOURNAL = {Arch. Ration. Mech. Anal.},
  FJOURNAL = {Archive for Rational Mechanics and Analysis},
    VOLUME = {233},
      YEAR = {2019},
    NUMBER = {3},
     PAGES = {975--1025},
      ISSN = {0003-9527,1432-0673},
   MRCLASS = {81Q05 (76Y05 82D37)},
  MRNUMBER = {3961293},
MRREVIEWER = {Fucai\ Li},
       DOI = {10.1007/s00205-019-01373-w},
       URL = {https://doi.org/10.1007/s00205-019-01373-w},
}

@article {Pecher2012,
    AUTHOR = {H. Pecher},
     TITLE = {Unconditional global well-posedness for the 3{D}
              {G}ross-{P}itaevskii equation for data without finite energy},
   JOURNAL = {Nonlinear Differ. Equ. Appl.},
    VOLUME = {20},
      YEAR = {2013},
    NUMBER = {6},
     PAGES = {1851--1877},
      ISSN = {1021-9722,1420-9004},
   MRCLASS = {35Q55 (35B30 37L50)},
  MRNUMBER = {3128697},
       DOI = {10.1007/s00030-013-0233-2},
       URL = {https://doi.org/10.1007/s00030-013-0233-2},
}

@phdthesis{Hientzsch,
    AUTHOR = {L.E. Hientzsch},
    SCHOOL = {Gran Sasso Science Insitute},
     TITLE = {Nonlinear Schrödinger equations and quantum fluids non vanishing at infinity: incompressible limit and quantum vortices},
      YEAR = {2019},
}

@article{Antonelli2023,
       AUTHOR = {P. Antonelli and L.E. Hientzsch and P. Marcati},
        TITLE = {Finite energy well-posedness for nonlinear Schr{\"o}dinger equations with non-vanishing conditions at infinity}, 
      JOURNAL = {arXiv e-prints},
         YEAR = {2023},
       EPRINT = {2301.00751},
          DOI = {https://doi.org/10.48550/arXiv.2301.00751},
          URL = {https://arxiv.org/abs/2301.00751},
ARCHIVEPREFIX = {arXiv},
 PRIMARYCLASS = {math.AP},
}

@article {AntonelliMarcati2010,
    AUTHOR = {P. Antonelli and P. Marcati},
     TITLE = {The quantum hydrodynamics system in two space dimensions},
   JOURNAL = {Arch. Ration. Mech. Anal.},
  FJOURNAL = {Archive for Rational Mechanics and Analysis},
    VOLUME = {203},
      YEAR = {2012},
    NUMBER = {2},
     PAGES = {499--527},
      ISSN = {0003-9527,1432-0673},
   MRCLASS = {82D15 (35Q35)},
  MRNUMBER = {2885568},
MRREVIEWER = {Mahendra\ Panthee},
       DOI = {10.1007/s00205-011-0454-7},
       URL = {https://doi.org/10.1007/s00205-011-0454-7},
}

@article {Antonelli2021,
    AUTHOR = {P. Antonelli and P. Marcati and H. Zheng},
     TITLE = {Genuine hydrodynamic analysis to the 1-{D} {QHD} system:
              existence, dispersion and stability},
   JOURNAL = {Comm. Math. Phys.},
  FJOURNAL = {Communications in Mathematical Physics},
    VOLUME = {383},
      YEAR = {2021},
    NUMBER = {3},
     PAGES = {2113--2161},
      ISSN = {0010-3616,1432-0916},
   MRCLASS = {35Q40 (76Y05 82D50)},
  MRNUMBER = {4244267},
       DOI = {10.1007/s00220-021-03998-z},
       URL = {https://doi.org/10.1007/s00220-021-03998-z},
}

@article{AntonelliMarcatiScandone,
       AUTHOR = {P. Antonelli and P. Marcati and R. Scandone},
        TITLE = {Existence and stability of almost finite energy weak solutions to the quantum Euler-Maxwell system},
      JOURNAL = {arXiv e-prints},
         YEAR = {2021},
          EID = {arXiv:2109.14588},
          DOI = {10.48550/arXiv.2109.14588},
ARCHIVEPREFIX = {arXiv},
       EPRINT = {2109.14588},
 PRIMARYCLASS = {math.AP},
}

@article {Antonelli2019,
    AUTHOR = {P. Antonelli},
     TITLE = {Remarks on the derivation of finite energy weak solutions to
              the {QHD} system},
   JOURNAL = {Proc. Amer. Math. Soc.},
  FJOURNAL = {Proceedings of the American Mathematical Society},
    VOLUME = {149},
      YEAR = {2021},
    NUMBER = {5},
     PAGES = {1985--1997},
      ISSN = {0002-9939,1088-6826},
   MRCLASS = {35Q35 (35Q40 76Y05)},
  MRNUMBER = {4232191},
       DOI = {10.1090/proc/14502},
       URL = {https://doi.org/10.1090/proc/14502},
}

@article {AntonelliMarcatiZheng2019,
    AUTHOR = {P. Antonelli and P. Marcati and H. Zheng},
     TITLE = {An intrinsically hydrodynamic approach to multidimensional
              {QHD} systems},
   JOURNAL = {Arch. Ration. Mech. Anal.},
  FJOURNAL = {Archive for Rational Mechanics and Analysis},
    VOLUME = {247},
      YEAR = {2023},
    NUMBER = {2},
     PAGES = {Paper No. 24, 58},
      ISSN = {0003-9527,1432-0673},
   MRCLASS = {76Y05},
  MRNUMBER = {4562812},
       DOI = {10.1007/s00205-023-01856-x},
       URL = {https://doi.org/10.1007/s00205-023-01856-x},
}

@article{Markovich,
    AUTHOR = {P. Markowich and J. Sierra},
    TITLE = {Non-uniqueness of weak solutions of the quantum-hydrodynamic system},
    JOURNAL = {Kinet. Relat. Models.},
    VOLUME = {12},
    NUMBER = {2},
    PAGES = {347--356},
    YEAR = {2019},
    ISSN = {1937-5093},
    DOI = {10.3934/krm.2019015},
    URL = {/article/id/a5bb2cfd-662b-43c0-a0ae-493e7c45c996},
}

@article{Bianchini,
    AUTHOR = {S. Bianchini},
     TITLE = {Exact integrability conditions for contangent vector fields},
   JOURNAL = {Manuscr. Math.},
      YEAR = {2023},
       DOI = {10.1007/s00229-023-01461-y},
       URL = {https://link.springer.com/article/10.1007/s00229-023-01461-y},
}

@article{Reddiger2023,
    AUTHOR = {M. Reddiger and B. Poirier},
     TITLE = {Towards a mathematical theory of the Madelung equations: Takabayasi’s quantization condition, quantum quasi-irrotationality, weak formulations, and the Wallstrom phenomenon},
   JOURNAL = {J. Phys. A: Math.},
       URL = {https://dx.doi.org/10.1088/1751-8121/acc7db},
      YEAR = {2023},
 PUBLISHER = {IOP Publishing},
    VOLUME = {56},
    NUMBER = {19},
     PAGES = {193001},
}

@article {Wallstrom1,
    AUTHOR = {T.C. Wallstrom},
     TITLE = {On the derivation of the {S}chr\"{o}dinger equation from
              stochastic mechanics},
   JOURNAL = {Found. Phys. Lett.},
  FJOURNAL = {Foundations of Physics Letters},
    VOLUME = {2},
      YEAR = {1989},
    NUMBER = {2},
     PAGES = {113--126},
      ISSN = {0894-9875,1572-9524},
   MRCLASS = {81C20 (81B05)},
  MRNUMBER = {994069},
MRREVIEWER = {Pekka\ Johannes\ Lahti},
       DOI = {10.1007/BF00696108},
       URL = {https://doi.org/10.1007/BF00696108},
}

@article {Wallstrom2,
    AUTHOR = {T.C. Wallstrom},
     TITLE = {On the initial value problem for the {M}adelung hydrodynamic
              equations},
   JOURNAL = {Phys. Lett. A},
  FJOURNAL = {Physics Letters. A},
    VOLUME = {184},
      YEAR = {1994},
    NUMBER = {3},
     PAGES = {229--233},
      ISSN = {0375-9601,1873-2429},
   MRCLASS = {81Q05},
  MRNUMBER = {1257367},
       DOI = {10.1016/0375-9601(94)90380-8},
       URL = {https://doi.org/10.1016/0375-9601(94)90380-8},
}

@article {Wallstrom3,
    AUTHOR = {T.C. Wallstrom},
     TITLE = {Inequivalence between the {S}chr\"{o}dinger equation and the
              {M}adelung hydrodynamic equations},
   JOURNAL = {Phys. Rev. A (3)},
  FJOURNAL = {Physical Review. A. Third Series},
    VOLUME = {49},
      YEAR = {1994},
    NUMBER = {3},
     PAGES = {1613--1617},
      ISSN = {1050-2947,1094-1622},
   MRCLASS = {81P99},
  MRNUMBER = {1269084},
MRREVIEWER = {Andr\`ej\ K.\ Kwa\'{s}niewski},
       DOI = {10.1103/PhysRevA.49.1613},
       URL = {https://doi.org/10.1103/PhysRevA.49.1613},
}

@article{Bohm,
    AUTHOR = {D. Bohm},
     TITLE = {A suggested interpretation of the quantum theory in terms of ``hidden''' variables},
   JOURNAL = {Phsy. Rev.},
    VOLUME = {85},
    NUMBER = {2},
      YEAR = {1952},
     PAGES = {166--179},
  PUBLISHER = {American Physical Society},
       DOI = {10.1103/PhysRev.85.166},
       URL = {https://link.aps.org/doi/10.1103/PhysRev.85.166}
}

@article{Nelson,
    AUTHOR = {E. Nelson},
     TITLE = {Derivation of the Schr\"odinger equation from Newtonian mechanics},
   JOURNAL = {Phys. Rev.},
    VOLUME = {150},
     ISSUE = {4},
     PAGES = {1079--1085},
  NUMPAGES = {0},
      YEAR = {1966},
 PUBLISHER = {American Physical Society},
       DOI = {10.1103/PhysRev.150.1079},
       URL = {https://link.aps.org/doi/10.1103/PhysRev.150.1079},
}

@article {Takabayasi,
    AUTHOR = {T. Takabayasi},
     TITLE = {On the formulation of quantum mechanics associated with
              classical pictures},
   JOURNAL = {Progr. Theoret. Phys.},
  FJOURNAL = {Progress of Theoretical Physics},
    VOLUME = {8},
      YEAR = {1952},
     PAGES = {143--182},
      ISSN = {0033-068X,1347-4081},
   MRCLASS = {81.0X},
  MRNUMBER = {52988},
MRREVIEWER = {K.\ M.\ Case},
       URL = {https://doi.org/10.1143/ptp/8.2.143},
       DOI = {10.1143/ptp/8.2.143}
}

\end{document}